\documentclass[a4paper,12pt]{article}
\usepackage{mathrsfs}
\usepackage{CJK}
\usepackage{color}
\usepackage{framed}
\usepackage{amsthm}
\usepackage{amsmath}
\usepackage{amssymb}
\usepackage{float}
\usepackage{graphicx}
\usepackage{graphics}
\usepackage{epstopdf}
\usepackage[caption=false]{subfig}
\usepackage{epsfig}
\usepackage{indentfirst}
\usepackage{cite}
\usepackage{geometry}
\usepackage[inline]{enumitem}
\numberwithin{equation}{section}
\newtheorem{theorem}{Theorem}[section]

\newtheorem{lemma}[theorem]{Lemma}
\newtheorem{corollary}[theorem]{Corollary}
\newtheorem{proposition}[theorem]{Proposition}
\newtheorem{remark}[theorem]{Remark}

\usepackage[]{hyperref}

\usepackage{fancyhdr}

\fancypagestyle{firstpage}{%
\fancyhead[L]{\small Accepted by 
\emph{Communications in Statistics - Theory and Methods}.}

}

\begin{document}
\title{Optimal Dividend Control with Transaction Costs under Exponential Parisian Ruin for a Refracted L\'{e}vy Risk Model
\thanks{This research is supported by National Natural Science Foundation of China (Grant No.11671204), China Scholarship Council (No.202206840089), and Postgraduate Research \& Practice Innovation Program of Jiangsu Province (Project No.KYCX22\_0392).}}

\author{Zhongqin Gao
{\thanks {School of Mathematics and Statistics, Nanjing University of Science and Technology, China. {\em E-mail:} zhongqingaox@126.com}}
\ and Yan Lv
{\thanks {School of Mathematics and Statistics, Nanjing University of Science and Technology, China. {\em E-mail:} lvyan1998@aliyun.com}}
\ and Jingmin He
{\thanks {School of Science, Tianjin University of Technology, China {\em E-mail:} nkjmhe\_2002@163.com}}}

\date{}
\maketitle
\thispagestyle{firstpage}

\noindent \hrulefill
\medskip
\begin{abstract}\maketitle\baselineskip=18pt
This paper concerns an optimal impulse control problem associated with a refracted L\'{e}vy process, involving the reduction of reserves to a predetermined level whenever they exceed a specified threshold. The ruin time is determined by Parisian exponential delays and limited by a lower ultimate bankrupt barrier. We initially obtained the necessary and sufficient conditions for the value function and the optimal impulse control policy. Given a candidate for the optimal strategy, the corresponding expected discounted dividend function is subsequently formulated in terms of the Parisian refracted scale function, which is employed to measure the expected discounted utility of the impulse control. Then, the optimality of the proposed impulse control is verified using the HJB inequalities, and a monotonicity-based criterion is established to identify the admissible region of optimal thresholds, which serves as the basis for the numerical computation of their optimal levels. Finally, we present applications and numerical examples related to Brownian risk process and Cram\'{e}r-Lundberg process with exponential claims, demonstrating the uniqueness of the optimal impulse strategy and exploring its sensitivity to parameters. 
{\bf{MR(2010) Subject Classification:}} {60G40, 93E20}
\end{abstract}

\newpage
\section{Introduction}

From the perspective of modern finance insurance, the risk theory has shown a preference for working with spectrally negative L\'{e}vy processes (SNLPs), which are stochastic processes with stationary and independent increments and no positive jumps. This class of models includes Brownian motion and Cram\'{e}r-Lundberg process as special cases. Recently, to accommodate the development of the insurance markets, a more general setting of a refracted L\'{e}vy risk process has been analyzed, that is a L\'{e}vy process whose dynamics change by subtracting a fixed linear drift of suitable size whenever the aggregate process is above a pre-specified level, where the level can be set by the insurer's solvency capital requirement. This process has been proved to exist and is characterized as a skip-free upward strong Markov process, as shown in Kyprianou and Loeffen\cite{Kyprianou20101}.

Classical ruin theory assumes that ruin occurs immediately when the surplus process first goes below zero. Recognizing that companies use financial reserves to minimize the chance of ruin and that a strict definition of ruin can result in lost potential profits, some researchers have focused on the Parisian ruin problem, where the company is allowed to operate under negative surplus for a predetermined period known as Parisian implementation delays. Such Parisian ruin was initially investigated by Dassios and Wu\cite{Dassios2008} in the context of the classical Cram\'{e}r-Lundberg model. Lately, Parisian ruin under various risk models has gained significant attention, with the two main research types being Parisian deterministic delays and Parisian stochastic delays. For Parisian deterministic delays, ruin occurs when the risk process stays negative continuously for a fixed period of time. This has been extensively studied, for example, Loeffen, Czarna and Palmowski\cite{Loeffen2013}, Lkabous, Czarna and Renaud\cite{Lkabous2017} and Loeffen, Palmowski and Surya\cite{Loeffen2018} discuss related ruin problems, while Czarna and Palmowski\cite{Czarna2014} and Yang, Sendova and Li\cite{Yang2020} address related dividend problems.

As is widely known in the applied probability literature, the deterministic approach has limitations in processing financial risk problems with Parisian ruin, leading to a natural need to define Parisian delays differently, such as the so-called Parisian stochastic delays. For this delay, it is proposed that each excursion of the surplus process below zero is accompanied by a positive independent and identically distributed (iid) random variable, and ruin occurs as soon as an excursion stays below zero longer than the predefined random time. More recent contribution to this research include, Landriault, Renaud, and Zhou\cite{Landriault2014} for mixed Erlang delays, and Baurdoux, Pardo, P\'{e}rez and Renaud\cite{Baurdoux20162}, Renaud\cite{Renaud2019} and Renaud\cite{Renaud2024} for exponential delays. Incidentally, some researchers analyzed Parisian implementation delays in dividend payments instead of applying them to recognizing ruin, see in Cheung and Wong\cite{Cheung2017}. In addition, it is unrealistic to assume that a company's surplus can decrease without bounds. Therefore, incorporating a lower ultimate bankruptcy barrier into the Parisian ruin model could serve as a better and more efficient risk measure. Recent work in this area can be found in Czarna and Renaud\cite{Czarna20162} and Frostig and Keren-Pinhasik\cite{Frostig2020}.

The optimization of dividend strategies is a critical focus in financial research, aimed at maximizing the expected present value of all dividends to be distributed to shareholders until the company is ruined or bankrupt. The concept of the optimal dividend problem was first proposed by De Finetti\cite{De1957} within a binomial model. Subsequently, many researchers have thoroughly researched this complex issue, as detailed in the works of Avanzi and Gerber\cite{Avanzi2008}, Albrecher and Thonhauser\cite{Albrecher2008}, Bayraktar, Kyprianou and Yamazaki\cite{Bayraktar2013}, and Avram, Palmowski and Pistorius\cite{Avram2015}. In company operations, including transaction costs in dividend payments is more practical and effective due to related expenses. Based on transaction cost economics, the dividend strategy evolves into an impulse control strategy, comprising a sequence of intervention times and control actions. Relevant studies include Loeffen\cite{Loeffen2009} on SNLPs, and Thonhauser and Albrecher\cite{Thonhauser2011} on the Cram\'{e}r-Lundberg process. Notably, Czarna and Kaszubowski\cite{Czarna2020} consider the optimality of impulse control problem in refracted L\'{e}vy model with Parisian deterministic delays and transaction costs.

The risk model to be discussed in this paper is given below. Let $X=\{X_t, t\geq0\}$ be a SNLP with characteristic triplet $(\gamma,\sigma,\nu)$ defined on the filtered probability space $(\Omega, \mathcal{F}, \mathbb{F}=\{\mathcal{F}_t, t\geq0\}, P)$, where $\sigma\geq0$ is the diffusion coefficient, $\gamma\in \mathbb{R}$ is the drift coefficient and $\nu$ is the L\'{e}vy measure that satisfies $\nu(-\infty, 0)=0$ and $\int_{0}^{\infty}(1\wedge z^2)\nu(\mathrm{d}z)< \infty$. The law of $X$ is denote by $P_x$ when it starts at $x\in \mathbb{R}$, and for convenience, $P$ is used instead of $P_0$. The associated expectation operators are denoted as $E_x$ and $E$, respectively. The refracted L\'{e}vy process $R=\{R_t: t\geq0\}$ is described by
\begin{eqnarray}\label{model1}
	R_t=X_t-\delta \int_{0}^{t}\mathbf{1}\{R_s\geq b\}\mathrm{d}s,\ \ \ \ t\geq 0, \ \ \delta\geq 0,
\end{eqnarray}
where $b\geq0$ is a threshold level, $\delta\geq0$ is a refraction parameter and $\mathbf{1}\{A\}$ is the indicator function of event $A$.
Economically, the threshold $b\geq 0$ represents a solvency or capital management level (e.g., a regulatory intervention threshold) at which the insurer switches between operating regimes, while the refraction parameter $\delta\geq 0$ quantifies the strength of the adjustment made in the high-surplus regime. When $R_t\geq b$, the company is viewed as well capitalised and the surplus is modified at rate $\delta$ (e.g. through active investment, intensified risk mitigation, or an additional dividend or reinsurance outflow), whereas for $R_t<b$ the surplus follows the baseline risk process $Y=\{Y_t=X_t-\delta t, t\geq0\}$ without such adjustments. Thus $b$ encodes the trigger at which capital management actions are activated, and $\delta$ measures the intensity of these actions, allowing the model to capture a range of investment, dividend, and restructuring policies within a unified framework.

The ruin discussed here includes an exponential implementation delay and a lower ultimate bankrupt barrier, as detailed below. Assume that each time the process $R$ down-crosses zero, it is accompanied by an independent exponentially distributed random variable $\xi$ with parameter $m\geq0$. Additionally, we set a lower ultimate bankrupt barrier at $-l<0$. Ruin occurs either at the first time that the exponential delay expires before the surplus becomes positive, or at the first time that the surplus down-crosses $-l$, whichever occurs first. Formally, the time of Parisian ruin with the ultimate bankrupt barrier is denoted by
\begin{align*}
&T:=\tau_p\wedge k_{-l}^-,
\end{align*}
with $\tau_p:=\inf\{t\geq 0: R_t<0\ \mbox{and} \ t>g_t+e_m^{g_t}\}$, where $g_t=\sup\{0\leq s< t: R_s\geq 0\}$ is the left-end point of a negative excursion, $e_m^{g_t}$ represents the Parisian exponential delay related to $g_t$ and is independent of $R$, and $k_{-l}^-$ is the first time to down-cross $-l$ for $R$.

Assume a company with risk process $R$ and ruin time $T$ pays dividends to its shareholders according to a dividend strategy $\pi=\{D_t^{\pi}\}_{t\geq 0}$, where $D_t^{\pi}$ is the cumulative dividend amount up to time $t\geq 0$, defined as $D_t^{\pi}=\sum_{0\leq s< t}\Delta D_s^{\pi}$, with $\Delta D_s^{\pi}=D_{s}^{\pi}-D_{s-}^{\pi}$ representing the jump of the dividend process $\{D_t^{\pi}\}_{t\geq 0}$ (which is non-decreasing, c\`{a}dl\`{a}g, $\mathbb{F}$-adapted, starts at zero, and is a pure jump process) at time $s$, and $D_0^{\pi}=0$. Then, we redefine the controlled risk process $U^{\pi}=\{U_t^{\pi}\}_{t\geq 0}$ as
\begin{eqnarray*}
	U_t^{\pi}=
X_t-\delta \int_{0}^{t}\mathbf{1}\{U_s^{\pi}\geq b\}\mathrm{d}s -D_t^{\pi},
\end{eqnarray*}
with $U_0^{\pi}=x$. The time of ruin for $U^{\pi}$ is defined by $T^{\pi}:=\tau_p^{\pi}\wedge k_{-l}^{\pi}$, where $\tau_p^{\pi}:=\inf\{t\geq 0: U_t^{\pi}<0\ \mbox{and} \ t>g_t^{\pi}+e_m^{g_t^{\pi}}\}$ with $g_t^{\pi}=\sup\{0\leq s< t: U_s^{\pi}\geq 0\}$, and $k_{-l}^{\pi}$ is the first time to down-cross $-l$ for $U^{\pi}$. Moreover, incorporating transaction costs into the dividend strategy $\pi$, meaning that each dividend payment inevitably produces a fixed transaction fee $\beta>0$, leads to the above dividend strategy being referred to as an impulse control strategy. This strategy supports the assumption that the dividend process is purely a jump process. The expected discounted utility of the impulse control is measured by the expected discounted dividend function, as defined below, for $x\in [-l, \infty)$,
\begin{align}\label{funcvpix}
V_{\pi}(x)
&=E_x\big[\int_0^{T^{\pi}} e^{-q t} \mathrm{d}(D_t^{\pi}-\sum_{0\leq s\leq t}\beta \mathbf{1}\{\Delta D_s^{\pi}>0\})\big]\nonumber \\
&=E_x\big[\sum_{0\leq t\leq T^{\pi}}e^{-q t}(\Delta D_t^{\pi}-\beta \mathbf{1}\{\Delta D_t^{\pi}>0\})\big],
\end{align}
where $q\geq0$ is the discount factor. A strategy $\pi$ is called admissible if $\Delta D_t^{\pi} \geq 0$ and, whenever a dividend is paid, the jump size exceeds a minimum threshold $\beta$, that is, $\Delta D_t^{\pi} >\beta \mathbf{1}\{\Delta D_t^{\pi}>0\}$ for any $t\in [0, T^{\pi})$, and the surplus does not fall below zero due to dividend payments, i.e.
\begin{align}\label{admiss}
&U_{t-}^{\pi}-\Delta D_t^{\pi}\geq 0.
\end{align}
Let $\Pi$ denote the set of all admissible dividend strategies. Our main goal is to find the value function $V_*$ defined by
\begin{align}\label{optpro}
&V_*(x)=\sup_{\pi\in\Pi}V_{\pi}(x),
\end{align}
and the optimal impulse control $\pi^*\in\Pi$ such that $V_{\pi^*}(x)=V_*(x)$ for all $x\in [-l, \infty)$. This problem we refer to as the impulse control problem, and researching it generalizes the ordinary-ruin concept, balances solvency and profitability, and provides insights into the behavior of the optimal impulse strategy with model parameters. Prior to the main analysis, we present an alternative but equivalent formulation of the function $V_{\pi}$, its proof is given in Appendix \ref{propequexp22}.
\begin{proposition}\label{propequexp}
The function $V_{\pi}$ defined in \eqref{funcvpix} can be equivalently expressed as
\begin{align}\label{vj1j200}
V_{\pi}(x)
&=E_x\big[\int_0^{\tau_p^{\pi} \wedge k_{-l}^{\pi}} e^{-q t} \mathrm{d} D_{\beta}^{\pi}(t)\big]
=E_x\big[\int_0^{k_{-l}^{\pi}} e^{-q t - L^{\pi}(t)} \mathrm{d} D_{\beta}^{\pi}(t)\big],
\end{align}
where $D_{\beta}^{\pi}(t):=\sum_{0\leq s\leq t} \Delta D_s^{\pi}-\beta \mathbf{1}\{\Delta D_s^{\pi}>0\}$ and $L^{\pi}(t) := \int_0^t m \mathbf{1}\{U_s^{\pi} <0\} \mathrm{d} s$.

\end{proposition}

The rest of this paper is organized as follows. In Section \ref{sec 2}, we introduce the scale functions for the SNLP and its refracted counterpart, and provide some known results related to the exit problem. In Section \ref{sec 3}, we first provide semi-explicit expression for the exit problem with Parisian exponential delays and a lower bankruptcy barrier, then derive the necessary and sufficient conditions for the value function, which also apply to the optimality of the impulse control strategy.
Section \ref{sec 4} discusses the optimality of the so-called impulse strategy $\pi_{(c_1, c_2)}$, which intend to reduce the surplus to a certain level $c_1$ whenever it exceeds another level $c_2$.
In Section \ref{sec 5}, numerical examples involving the refracted Brownian motion and the refracted Cram\'{e}r-Lundberg process with exponential claims are presented. We prove the existence and uniqueness of the optimal impulse strategy for these processes and summarize several valuable conclusions.

\section{Preliminaries}\label{sec 2}
\label{sec:preliminaries}
For a risk process modelled by a SNLP, most quantities of interest in risk theory are often expressed using the so-called scale functions. In this section, the definitions and properties of scale functions for SNLPs and the refracted L\'{e}vy process are given. Further, some arguments concerning fluctuation identities involved in the exit problem are reviewed, as they are necessary to address the risk issues.
Lastly, we provide several auxiliary functions for use in the subsequent results.

\subsection{Scale function}
Recall that $X=\{X_t, t\geq 0\}$ is a SNLP with L\'{e}vy triplet $(\gamma,\sigma,\nu)$. To avoid degenerate cases, the case where $X$ has monotone paths should be excluded. Its Laplace exponent $\psi: [0,\infty)\rightarrow \mathbb{R}$ exists and is defined by, for all $\lambda\geq0$,
$$\psi(\lambda):=\log(E[e^{\lambda X_{1}}])=\frac{1}{2}\sigma^2\lambda^2+\gamma \lambda +\int_{0}^{\infty}(e^{-\lambda z}-1+\lambda z\mathbf{1}\{0<z<1\})\nu(\mathrm{d}z).$$
Notice that $\psi{'}(0+)=E[X_1]$. The process $X$ has paths of bounded variation (BV) if and only if $\sigma=0$ and $\int_{0} ^{1}z\nu(\mathrm{d}z)<\infty$, for this case, $X$ can be written as $X_t=\mu t-S_t$, where $\mu=\gamma+\int_{0} ^{1}z\nu(\mathrm{d}z)>0$ the so-called drift of $X$, and $S=\{S_t, t\geq0\}$ is a driftless subordinator such as a gamma process or a compound Poisson process with positive jumps. Conversely, if $\sigma>0$, then $X$ has paths of unbounded variation (UBV). For $q\geq0$, define $\phi(q):=\sup\{\lambda\geq 0: \psi(\lambda)=q\}$. From Kyprianou\cite{Kyprianou20142} Chapter 8.1, if $\psi{'}(0+)\geq 0$, then $\phi(0)=0$, otherwise $\phi(0)>0$.

The $q$-scale function $W^{(q)}(x)$ of $X$ is defined on $[0,\infty)$ by its Laplace transform
\begin{eqnarray*}
	\int_{0} ^{\infty}e^{-\lambda x}W^{(q)}(x)\mathrm{d}x=\frac{1}{\psi(\lambda)-q},  \ \ \mbox{for}\ \lambda>\phi(q),
\end{eqnarray*}
which is unique, positive, strictly increasing and continuous for $x> 0$ and is further continuous for $q\geq 0$. We define $W^{(q)}(x)=0$ for all $x\in(-\infty,0)$, which enables to extend the definition domain to the entire real axis. Considering the possible atoms at the origin, it shows that $W^{(q)}(0+)=1/\mu$ if $X$ has paths of BV, and if not, $W^{(q)}(0+)=0$ with the definition $W^{(q)}(0+):=\lim_{x\downarrow0}W^{(q)}(x)$, as well as $W^{(q)}{'}(0+)=(\nu(0, \infty)+q)/\mu^2$ when $\sigma=0$ and $\nu(0, \infty)<\infty$, and $W^{(q)}{'}(0+)=2/\sigma^2$ when $\sigma>0$.
For discussion on the differentiability of the scale function, see Kuznetsov, Kyprianou and Rivero\cite{Kuznetsov2012}.

Define $Y=\{Y_t=X_t-\delta t, t\geq0\}$ is a SNLP with L\'{e}vy triplet $(\gamma-\delta,\sigma,\nu)$, its Laplace exponent is given by $\psi(\lambda)-\delta \lambda$. When $X$ has paths of BV, we assume $0\leq \delta<\mu$, meaning a part of drift of $X$ will be retained. In fact, $X$ and $Y$ share many properties except for those affected by the linear part of the L\'{e}vy process. The scale function of $Y$ on $[0, \infty)$, denoted as $\mathbb{W}^{(q)}$, is defined by
\begin{eqnarray*}
	&&\int_{0} ^{\infty}e^{-\lambda x}\mathbb{W}^{(q)}(x)\mathrm{d}x=\frac{1}{\psi(\lambda)-\delta\lambda-q},  \ \ \mbox{for}\ \lambda>\varphi(q),
\end{eqnarray*}
where $\varphi(q):=\sup\{\lambda\geq0: \psi(\lambda)-\delta\lambda=q\}$.

For the refracted L\'{e}vy process $R=\{R_t, t\geq 0\}$, it satisfies the condition of positive security loading to ensure that the probability of ruin is strictly less than one, namely, $E[Y_1]=\psi{'}(0+)-\delta\geq 0$. The scale function of $R$, represented by $w^{(q)}(x;a)$, is defined as, for $q\geq0$ and $x, a\in\mathbb{R}$,
\begin{eqnarray}\label{1wqde}
	&&w^{(q)}(x;a):=W^{(q)}(x-a)+\delta \mathbf{1}{\{x\geq b\}}\int_{b} ^{x}\mathbb{W}^{(q)}(x-y)W^{(q)}{'}(y-a)\mathrm{d}y.
\end{eqnarray}
When $x<b$, one has $w^{(q)}(x;a)=W^{(q)}(x-a)$. We refer to
Kyprianou\cite{Kyprianou20142} for a recent overview of scale function.
We also recall the following useful identity, taken from Loeffen, Renaud and Zhou\cite{Loeffen2014}, for $q, p\geq 0$ and $x\in\mathbb{R}$,
\begin{eqnarray}\label{2}
	&&p\int_{0}^{x}W^{(q)}(x-y)W^{(q+p)}(y)\mathrm{d}y=W^{(q+p)}(x)-W^{(q)}(x).
\end{eqnarray}
Worth noting, $R$ is now a Feller process rather than a L\'{e}vy process, as it is no longer spatial homogeneous, as shown in Chapter 11 of Khoshnevisan and Schilling\cite{Khoshnevisan2016}. Let $\mathbb{C}^k(\mathbb{A})$ be the space of $k$-times continuously differentiable functions defined on the set $\mathbb{A}$.

\begin{proposition}\label{prop 2.1}
	In general, $w^{(q)}(\cdot; a)$ is a.e. continuously differentiable for any $q\geq 0$ and $a\in\mathbb{R}$. More precisely, if we assume $W^{(q)}(\cdot - a)\in \mathbb{C}^1((a, \infty))$ for $X$ which is of BV (a necessary and sufficient condition for this is that the L\'{e}vy measure has no atoms, see in Kyprianou, Rivero and Song\cite{Kyprianou2010}), then $w^{(q)}(\cdot; a)\in \mathbb{C}^1((a, \infty)\setminus\{ b\})$. In contrast, if $X$ which is of UBV, then $w^{(q)}(\cdot; a)\in \mathbb{C}^1((a, \infty))$. Moreover, if $X$ has a Gaussian component $\sigma>0$, then $w^{(q)}(\cdot; a)\in \mathbb{C}^2((a, \infty)\setminus \{b\})$.
From Czarna, P\'{e}rez, Rolski and Yamazaki\cite{Czarna2019}, its derivative satisfies $w^{(q)}{'}(x; a)=W^{(q)}{'}(x-a)$ for $x<b$, and for $x\geq b$,
	\begin{eqnarray*}
		w^{(q)}{'}(x; a)=\big(1+\delta \mathbb{W}^{(q)}(0)\big)W^{(q)}{'}(x-a)+\delta \mathbf{1}\{x\geq b\} \int_{b}^x \mathbb{W}^{(q)}{'}(x-y)W^{(q)}{'}(y-a)\mathrm{d}y.
	\end{eqnarray*}
\end{proposition}

\subsection{Fluctuation identity}

The exit problem is one of the most important issues studied in the theory of L\'{e}vy process, and some related results are introduced as follows. For any $a,c\in\mathbb{R}$, the first passage stopping times are denoted by
\begin{align*}
\tau_a^- & =\inf\{t>0: X_t<a\},\ \ \ \ \ \ \ \ \ \ \tau_c^+=\inf\{t>0: X_t\geq c\},\\
\upsilon_a^- & =\inf\{t>0: Y_t<a\},\ \ \ \ \ \ \ \ \ \ \upsilon_c^+=\inf\{t>0: Y_t\geq c\},\\
k_a^- & =\inf\{t>0: R_t<a\},\ \ \ \ \ \ \ \ \ \  k_c^+=\inf\{t>0: R_t\geq c\},
\end{align*}
with the convention $\inf\emptyset=\infty$. It is well known (see, Kyprianou and Loeffen\cite{Kyprianou20101}) that, for $q\geq0$ and $a\leq x, b\leq c$, the solutions to the two-sided exit problem for $X$, $Y$ and $R$ are provided, respectively, by
\begin{eqnarray}
&E_x[e^{-q\tau_c^+}\mathbf{1}\{\tau_c^+<\tau_a^-\}]=\frac{W^{(q)}(x-a)}{W^{(q)}(c-a)},\label{3}\\
&E_x[e^{-q\upsilon_c^+}\mathbf{1}\{ \upsilon_c^+<\upsilon_a^-\}]=\frac{\mathbb{W}^{(q)}(x-a)}{\mathbb{W}^{(q)}(c-a)},\label{4}\\
&E_x[e^{-qk_c^+}\mathbf{1}\{k_c^+<k_a^-\}]=\frac{w^{(q)}(x; a)}{w^{(q)}(c; a)}.\label{5}
\end{eqnarray}
To study the Laplace transform of the time to Parisian ruin with a bankruptcy barrier, we need to consider the discounted expectation of the joint distribution of the time to down-cross a level and the scale function evaluated at the surplus value of the down-shoot.
From Kyprianou\cite{Kyprianou20142} and Loeffen, Renaud and Zhou\cite{Loeffen2014}, some technical fluctuation identities are presented for later use.
For $q, p\geq 0$, $a\geq 0$ and $-a\leq x\leq b$, we have
\begin{align}\label{6} &E_x[e^{-q\tau_0^-}\mathbf{1}\{{\tau_0^-<\tau_b^+\}}W^{(q+p)}(X_{\tau_0^-}+a)]
=g^{(q+p,q)}(x,a)-\frac{W^{(q)}(x)}{W^{(q)}(b)}g^{(q+p,q)}(b,a),
\end{align}
where
\begin{align}\label{7}
	g^{(q+p,q)}(x,a)
&=W^{(q+p)}(x+a)-p\int_0^{x}W^{(q)}(x-y)W^{(q+p)}(y+a)\mathrm{d}y \nonumber \\
&=W^{(q)}(x+a)+p\int_0^{a}W^{(q)}(x+a-y)W^{(q+p)}(y)\mathrm{d}y.
\end{align}
For $q, p\geq 0$, $b\geq 0$ and $b\leq x \leq c$, we also get
\begin{align}\label{8}
	& E_x[e^{-q\upsilon_b^-}\mathbf{1}\{{\upsilon_b^-<\upsilon_c^+\}}W^{(q+p)}(Y_{\upsilon_b^-})]
	=h^{(q+p,q)}(x,b)-
	\frac{\mathbb{W}^{(q)}(x-b)}{\mathbb{W}^{(q)}(c-b)}h^{(q+p,q)}(c,b),
\end{align}
where
$h^{(q+p,q)}(x,b)=w^{(q+p)}(x;0)-p\int_b^{x}\mathbb{W}^{(q)}(x-y)w^{(q+p)}(y;0)\mathrm{d}y$.

Lastly, several auxiliary functions are provided. For $q, p\geq 0$ and $x, a\in\mathbb{R}$, the Parisian refracted $(q, p)$-scale function is defined as
\begin{align}\label{9}
	\vartheta^{(q+p,q)}(x, a)
&:= w^{(q)}(x;a)+ p \int_0^{-a} w^{(q)}(x;a+y)W^{(q+p)}(y)\mathrm{d}y.
\end{align}
For $q\geq0$, $x, c\geq b\geq0$, we have
\begin{align}\label{10}
	&\varpi^{(q)}(x,b,c):=\frac{\mathbb{W}^{(q)}(x-b) w^{(q)}(c; 0)}{\mathbb{W}^{(q)}(c-b) W^{(q)}(b)}.
\end{align}

Note that the term 'sufficiently smooth' is used here in a slightly weaker sense, allowing for finitely many isolated discontinuities in the first or second derivative.
\begin{proposition}\label{prop 4.5}(Smoothness)
For any $q, p\geq 0$ and $a\in\mathbb{R}$, when $X$ is of BV, if we assume $W^{(q)}(\cdot)\in \mathbb{C}^1((0, \infty))$, then the scale function $\vartheta^{(q+p,q)}(\cdot, a)\in \mathbb{C}^1((a, \infty) \setminus \{0, b\})$, with sufficient smoothness for analytical purposes. When $X$ is of UBV, $W^{(q)}(\cdot)\in \mathbb{C}^1((0, \infty))$ intrinsically, and if $W^{(q)}{'}(\cdot)$ is assumed to be absolutely continuous on $(0,\infty)$ with the derivative $W^{(q)}{''}(\cdot)$ which is bounded on sets of the form $[1/n, n]$, $n\geq 1$, then $\vartheta^{(q+p, q)}(\cdot, a)$ is sufficiently smooth.
\end{proposition}

\section{Conditions for optimal impulse control}\label{sec 3}
\label{sec:impulse}
In this section, we will analyze the optimal conditions for the dividend strategy of
the refracted L\'{e}vy risk process, considering transaction costs, under Parisian ruin with an ultimate bankruptcy barrier.
First, semi-analytical expressions are provided for the exit problem with Parisian exponential delays and a lower bankruptcy barrier. Then, using standard Markovian arguments, we prove the Hamilton-Jacobi-Bellman (HJB) inequalities corresponding to the optimal dividend problem \eqref{optpro}, which represent the sufficient conditions for the dividend strategy to be optimal.

\subsection{Exit problem}

We derive the Laplace transform of the exit time up to Parisian ruin with the given exponential delay $\xi\sim \exp(m)$ and lower ultimate bankruptcy barrier $-l<0$, which impacts the reasoning of the value function, with the proof given in Appendix \ref{prop3112}.
\begin{proposition}\label{prop 3.1}
	For $q, m\geq 0$, $l>0$, $c, b\geq 0$ and $x\in[-l, c)$, we have
	\begin{eqnarray}\label{11}
		&&E_x[e^{-q k_{c}^+}\mathbf{1}\{k_{c}^+<T\}]=\frac{\vartheta^{(q+m,q)}(x,-l)}{\vartheta^{(q+m,q)}(c,-l)}.
	\end{eqnarray}
\end{proposition}

On the basis of \eqref{11}, we now derive several limiting cases. A detailed proof of Corollary \ref{cor00} is given in Appendix \ref{cor2100}, the other corollaries follow by similar arguments.

\begin{corollary}\label{cor00}
Letting $l\rightarrow \infty$, ruin is governed solely by the Parisian exponential delay, i.e. $T=\tau_p$. For $x\in (-\infty,c)$, we have
\begin{eqnarray}\label{ex26}
&&E_x[e^{-q k_c^+}I\{k_c^+<\tau_p\}]=\frac{F^{(q+m,q)}(x)}{F^{(q+m,q)}(c)},
\end{eqnarray}
where $H^{(q+m,q)}(x) := e^{\phi(q+m)x}\big(1-m \int_0^x e^{-\phi(q+m)(x-y)} W^{(q)}(x-y)\mathrm{d}y\big)$ and
\begin{align*}
&F^{(q+m,q)}(x):=H^{(q+m,q)}(x) +\delta \mathbf{1}{\{x\geq b\}}\\
&\cdot \int_b^x \mathbb{W}^{(q)}(x-y) \Big(\big(\phi(q+m)-m W^{(q)}(0)\big) e^{\phi(q+m)y}
-m\int_0^y e^{\phi(q+m)z} W^{(q)}{'}(y-z)\mathrm{d}z\Big)\mathrm{d} y. \nonumber
\end{align*}
\end{corollary}

\begin{corollary}\label{cor010}
For $m=0$, the model \eqref{model1} is subject to no restriction on the time it may
spend below zero, and ruin is declared at the first down-crossing of the lower
bankruptcy barrier $-l$, i.e. $T = k_{-l}^-$. For $x\in[-l, c)$, we have
\[
E_x\bigl[e^{-q k_c^+}\mathbf{1}{\{k_c^+<k_{-l}^-\}}\bigr]
= \frac{w^{(q)}(x; -l)}{w^{(q)}(c; -l)},
\]
which is consistent with \eqref{5}.
In particular, as $l\to 0$ (equivalently, $m\to\infty$), ruin is defined in the classical sense as the first passage time below zero, i.e. $T=k_0^-$.
\end{corollary}

\begin{corollary}\label{cor22}
For $\delta=0$, the model \eqref{model1} reduces to SNLP $X$. Then, for $x\in [-l,c)$,
\begin{eqnarray}\label{exit00}
&&E_x[e^{-q \tau_c^+}\mathbf{1}\{\tau_c^+<T\}]=\frac{g^{(q+m,q)}(x, l)}{g^{(q+m,q)}(c, l)},
\end{eqnarray}
which is consistent with Proposition~3.1 of Frostig and Keren-Pinhasik\cite{Frostig2020}.
Further, letting $l\rightarrow \infty$, we have $T= \tau_p$, and for $x\in(-\infty, c)$,
\begin{eqnarray*}
&&E_x[e^{-q \tau_c^+}\mathbf{1}\{\tau_c^+<{\tau}_{p}\}]
=\frac{H^{(q+m,q)}(x)}
{H^{(q+m,q)}(c)},
\end{eqnarray*}
which recovers the result in Baurdoux, Pardo, P\'{e}rez and Renaud\cite{Baurdoux20162} and Lemma 2.1 of Landriault, Renaud, and Zhou\cite{Landriault2014}. Moreover, for $m=0$, we have $T=\tau_{-l}^-$, and \eqref{exit00} coincides with \eqref{3}.
\end{corollary}

\subsection{Necessary and sufficient conditions for optimal impulse control}
\begin{lemma}\label{lem 3.2}
	For $x\geq y\geq 0$, $x-y>\beta \mathbf{1}\{x>y\}$ and $l>0$, the value function $V_*$ satisfies
	\begin{eqnarray*}
		x-y-\beta\mathbf{1}\{x>y\} \leq V_*(x)-V_*(y)\leq (1-\frac{\vartheta^{(q+m, q)}(y,-l)}{\vartheta^{(q+m, q)}(x,-l)})V_*(x),
	\end{eqnarray*}
	where $\vartheta^{(q+m, q)}(\cdot, -l)$ is defined in \eqref{9}. Further, $V_*(x)$
is increasing and continuous.
\end{lemma}
\begin{proof}
To prove the first inequality, for $\epsilon>0$ and initial capital $y\geq 0$, we consider a $\epsilon$-optimal strategy $\pi_{y*}$ such that $V_*(y)\leq V_{\pi_{y*}}(y)+\epsilon$. For initial capital $x$, using the admissible strategy $\tilde{\pi}_x$ in which there is an initial dividend payment with amount $x-y$ and followed by the application of strategy $\pi_{y*}$. Then we have $V_*(x)\geq V_{\tilde{\pi}_x}(x)=x-y-\beta\mathbf{1}\{x>y\}+V_{\pi_{y*}}(y)\geq x-y-\beta\mathbf{1}\{x>y\}+V_{*}(y)-\epsilon$. Let $\epsilon\downarrow 0$, one gets $V_*(x)-V_*(y)\geq x-y-\beta\mathbf{1}\{x>y\}$.
Assuming $V_*$ is differentiable at $x$, since $x-y>\beta\mathbf{1}\{x>y\}$,
the mean value theorem leads to $V_*'(x)>0$ for any $x\geq 0$, i.e. $V_*(x)$ is increasing. (In fact, Proposition~\ref{corol 4.3} establishes that $V_{*}$ is sufficiently smooth on $[-l, \infty)$ in the weaker sense of smoothness defined in Proposition~\ref{prop 4.5}, and the regularity of $V_*(x)\in \mathbb{C}^1((-l, \infty)\setminus \{0, b\})$ holds in both the BV and UBV cases.)

Next, we prove the second inequality. For initial capital $x$, the $\epsilon$-optimal strategy $\pi_{x*}$ also satisfies $V_*(x)\leq V_{\pi_{x*}}(x)+\epsilon$. For initial capital $y$, we adopt an admissible strategy $\tilde{\pi}_y$ in which there is no dividend payment before the process $R$ first up-crosses the level $x$, and followed by strategy $\pi_{x*}$. By \eqref{11}, we have
	\begin{align*}
		V_*(y) & \geq V_{\tilde{\pi}_y}(y)=E_y[e^{-q k_x^+}\mathbf{1}\{k_x^+<T\}]V_{\pi_{x*}}(x)
		=\frac{\vartheta^{(q+m,q)}(y,-l)}{\vartheta^{(q+m,q)}(x,-l)}V_{\pi_{x*}}(x)\\
& \geq \frac{\vartheta^{(q+m,q)}(y,-l)}{\vartheta^{(q+m,q)}(x,-l)}(V_{*}(x)-\epsilon).
	\end{align*}
	By letting $\epsilon\downarrow 0$, we obtain the second inequality as well as the continuity of the value function.
\end{proof}

For $x\in \mathbb{R}$, define the operator $\mathcal{A}$ by
\begin{align}\label{operator}
	\mathcal{A} f(x)
& = (\gamma-\delta \mathbf{1}\{x>b\})f{'}(x)+\frac{\sigma^2}{2} f{''}(x) \nonumber \\
& \ \ \ + \int_0^{\infty}(f(x-z)-f(x)+ f{'}(x)z \mathbf{1}\{0<z<1\})\nu(\mathrm{d}z),
\end{align}
where $f: \mathbb{R} \rightarrow \mathbb{R}$ is a sufficiently smooth function for which $\mathcal{A} f$ is well-defined.
Applying the Bouleau-Yor formula (see Bouleau and Yor\cite{Bouleau1981}) for $X$ is of BV and Meyer-It\^{o} formula (extend second derivative, see Theorem IV.71, Protter\cite{Protter2005}) for UBV case, we can deduce $f(U_{t}^{\pi})$ satisfies the following equation, and its proof is deferred to Appendix \ref{appen}. $L^{\pi}(t)$ is as defined in Proposition~\ref{propequexp}.
\begin{lemma}\label{lemdiff}
For $t\geq 0$, we have
\begin{align}\label{dududu}
e^{-q t -L^{\pi}(t)}f(U_{t}^{\pi})
&= f(x)+ \int_0^{t} e^{-q s -L^{\pi}(s)} (\mathcal{A} -q - m \mathbf{1}\{U_{s-}^{\pi}<0\}) f(U_{s-}^{\pi})\mathrm{d}s\\
&\ \ \
+\sum_{0\leq s\leq t}e^{- q s -L^{\pi}(s)}\big(f(U_{s-}^{\pi} + \Delta X_s - \Delta D_s^{\pi})-f(U_{s-}^{\pi} + \Delta X_s) \big) + \mathcal{M}_{t},\nonumber
\end{align}
where $\Delta X_t=X_t-X_{t-}$ represents the possible negative jump of $X$ at time $t$,
$\Delta D_t^{\pi}= D_t^{\pi}-D_{t-}^{\pi}$ denotes the possible pay-out dividend at time $t$,
and $\{\mathcal{M}_t\}_{t\geq 0}$ is an $\{\mathcal{F}_t, t\ge0\}$-adapted local-martingale  defined in \eqref{martinm}.
\end{lemma}
Notice that $\Delta U^{\pi}_t=\Delta X_t -\Delta D^{\pi}_t$ and the identity \eqref{dududu} also hold if $\Delta X_t\cdot\Delta D^{\pi}_t=0$.

\begin{lemma}\label{lem 3.4}(Verification Lemma)
\begin{enumerate}[label=(\roman*)]
\item
Suppose $V$ is sufficiently smooth on $[-l, \infty)$, that is, the first derivative for X of BV or the second derivative for X of UBV has at most finitely many isolated discontinuities, satisfies the following HJB inequalities
	\begin{align}
		&(\mathcal{A}-q - m \mathbf{1}\{x<0\}) V(x)\leq 0 && \mbox{for}  \ x\in [-l, \infty); \label{21}\\
&V(x)-V(y)\geq x-y-\beta \mathbf{1}\{x>y\}  && \mbox{for} \  x\geq y\geq 0 \ \mbox{and} \ x-y>\beta \mathbf{1}\{x>y\}.\label{22}
	\end{align}
Then, $V(x)\geq V_*(x)$ for a.e. $x\in [-l, \infty)$.
\item
If $\hat{\pi}$ is an admissible dividend strategy with the associated expected discounted dividend function, $V_{\hat{\pi}}$, satisfying the smoothness condition and the HJB inequalities \eqref{21}-\eqref{22} in (i), then $V_{\hat{\pi}}(x)=V_*(x)$.
\end{enumerate}
\end{lemma}
\begin{proof}
We only need to prove that for all $\pi\in\Pi$ and $x\in [-l, \infty)$, $V(x)\geq V_{\pi}(x)$.
Fix $\pi\in\Pi$ and define the sequence of stopping times $\{\tau_n\}_{n\in \mathbb{N}}$ as
\begin{align*}
&\tau_n:=\inf\{t\geq0: U_t^{\pi}>n \ \mbox{or} \ U_t^{\pi} < -l +\frac{1}{n} \}.
\end{align*}
From Schilling\cite{Schilling1998} and Schnurr\cite{Schnurr2012} it follows that $U^{\pi}$ is a semi-martingale.
Since $V$ is sufficiently smooth on $[-l, \infty)$ (for $X$ of BV or UBV, respectively), and $\{e^{-q(t\wedge \tau_n) - L^{\pi}(t\wedge \tau_n)}$ $V(U_{t\wedge \tau_n}^{\pi})\}_{t\geq 0}$ is the stopped process, by Lemma \ref{lemdiff}, we conclude that under $P_x$,
\begin{align*}
& e^{-q(t\wedge \tau_n)- L^{\pi}(t\wedge \tau_n)}V(U_{t\wedge \tau_n}^{\pi})
 = V(x)+\int_0^{t\wedge \tau_n}e^{-q s - L^{\pi}(s)}(\mathcal{A}-q - m \mathbf{1}\{U_{s-}^{\pi}<0\}) V(U_{s-}^{\pi})\mathrm{d}s \nonumber \\
& \ \ \ \ \ \ \ \ \ \ \ \ \ \ \ \ \ \ \  +\sum_{0\leq s\leq t\wedge \tau_n}e^{-q s - L^{\pi}(s)}\big(V(U_{s-}^{\pi} +\Delta X_s-\Delta D_s^{\pi})-V(U_{s-}^{\pi}+\Delta X_s)\big)+\mathcal{M}_{t\wedge \tau_n}.
\end{align*}
Since dividend payments do not reduce the surplus below zero, and $\Delta D_t^{\pi}> \beta \mathbf{1}\{\Delta D_t^{\pi}>0\}$ for any $t\geq 0$, by \eqref{22}, we have
\begin{align*}
V(U_{s-}^{\pi}+\Delta X_s -\Delta D_s^{\pi})-V(U_{s-}^{\pi}+\Delta X_s)
&=-\big(V(U_{s-}^{\pi}+\Delta X_s)-V(U_{s-}^{\pi} +\Delta X_s -\Delta D_s^{\pi})\big)\\
& \leq
-(\Delta D_s^{\pi}-\beta\mathbf{1}\{\Delta D_s^{\pi}>0\}) \leq 0.
\end{align*}
Following this, by combining \eqref{21}, we obtain
\begin{align*}
V(x)
&=e^{-q(t\wedge \tau_n)- L^{\pi}(t\wedge \tau_n)}V(U_{t\wedge \tau_n}^{\pi})-\int_0^{t\wedge \tau_n}e^{-q s - L^{\pi}(s)}(\mathcal{A}-q- m \mathbf{1}\{U_{s-}^{\pi}<0\}) V(U_{s-}^{\pi})\mathrm{d}s\\
&\ \ \  -\sum_{0\leq s\leq t\wedge \tau_n}e^{-q s - L^{\pi}(s)}\big(V(U_{s-}^{\pi} +\Delta X_s -\Delta D_s^{\pi})-V(U_{s-}^{\pi}+\Delta X_s)\big)- \mathcal{M}_{t\wedge \tau_n}\\
&\geq
e^{-q(t\wedge \tau_n)- L^{\pi}(t\wedge \tau_n)}V(U_{t\wedge \tau_n}^{\pi})
+\sum_{0\leq s\leq t\wedge \tau_n}e^{-q s - L^{\pi}(s)}\big(\Delta D_s^{\pi}-\beta\mathbf{1}\{\Delta D_s^{\pi}>0\}\big) - \mathcal{M}_{t\wedge \tau_n}.
\end{align*}
Note that ${\tau_n}\xrightarrow{n\uparrow\infty} k_{-l}^{\pi}$ for $P_x$-a.s. Taking expectation on both sides of the inequality above
and letting $t, n\uparrow\infty$, since $V\geq 0$ on $[-l, \infty)$, by the monotone convergence theorem, we have
	\begin{align*}
		 & V(x) \geq \lim_{t,n\uparrow\infty}E_x\Big[\int_0^{t\wedge \tau_n} e^{-q s - L^{\pi}(s)}\big(\mathrm{d}D^{\pi}_{s}-\sum_{0\leq r\leq s}\beta \mathbf{1}\{\Delta D_r^{\pi}>0\}\big) -\mathcal{M}_{t\wedge \tau_n} \nonumber
\\ & \ \ \ \ \ \ \ \ \ \ \ \ \ \ \ \ \ \ \ \ \ \ \ \ +e^{-q(t\wedge \tau_n)- L^{\pi}(t\wedge \tau_n)}V(U_{t\wedge \tau_n}^{\pi})\Big]\\
&\ \ \ \ \ =E_x\Big[\int_0^{k_{-l}^{\pi}} e^{-q s - L^{\pi}(s)} \big(\mathrm{d}D^{\pi}_{s}-\sum_{0\leq r\leq s}\beta \mathbf{1}\{\Delta D_r^{\pi}>0\}\big)
+\lim_{t,n\uparrow\infty}e^{-q(t\wedge \tau_n)- L^{\pi}(t\wedge \tau_n)}V(U_{t\wedge \tau_n}^{\pi})\Big]\nonumber\\
&\ \ \ \ \ \geq V_{\pi}(x).
	\end{align*}
By the arbitrariness of the strategy $\pi$, it follows that $V(x)\geq V_{\pi}(x)$ for all $\pi\in\Pi$ and a.e. $x\in [-l, \infty)$.
\end{proof}

\section{Optimal impulse control strategy}\label{sec 4}
We introduce a candidate for the optimal strategy in the dividend problem \eqref{optpro}, known as the impulse control strategy $\pi_{(c_1, c_2)}$, which represents an important type of strategy for impulse control problem.
Specifically, we set two levels $c_1$ and $c_2$ such that $c_1\geq0$ and $c_2>c_1+\beta$, and fix a set of the stopping times $\{\theta_k^{(c_1, c_2)}, k=1,2,\cdots\}$, where $\theta_k^{(c_1, c_2)}$ denotes the time at which the process $U^{(c_1, c_2)}$ exceeds $c_2$ for the $k$-th time, i.e. $\theta_1^{(c_1, c_2)}:=\inf\{t\geq 0: U_t^{(c_1, c_2)}>c_2\}$ and $\theta_k^{(c_1, c_2)}:=\inf\{t>\theta_{k-1}^{(c_1, c_2)}: U_t^{(c_1, c_2)}>c_2\}$ for $k\geq 2$.
The expected discounted dividend function and ruin time for the strategy $\pi_{(c_1, c_2)}$ are denoted as $V_{(c_1, c_2)}$ and $T^{(c_1, c_2)}$, respectively. We also denote $\tau_p^{\pi}$ and $k_{-l}^{\pi}$ by $\tau_p^{(c_1, c_2)}$ and $k_{-l}^{(c_1, c_2)}$, respectively. In summary, the impulse strategy $\pi_{(c_1, c_2)}$ is to reduce the risk process $U^{(c_1, c_2)}$ to $c_1$ whenever it exceeds level $c_2$. Here, $c_2-c_1>\beta$ ensures that shareholders receive dividends after paying the transaction costs, while $c_1\geq0$ is based on admissible constraint \eqref{admiss}.

\subsection{Expected discounted dividend function for impulse strategy $\pi_{(c_1, c_2)}$}
We first present an expression $V_{(c_1, c_2)}$ for a general impulse strategy $\pi_{(c_1, c_2)}$ with the ruin time $T^{(c_1, c_2)}$.
Recall that, $q\geq0$ represents the discount factor, $m\geq0$ is the parameter of the exponential delay, $-l<0$ is the ultimate bankrupt barrier, and $\beta>0$ denotes the transaction fee.
\begin{proposition}\label{prop 4.1}
For $c_1\geq0$, $c_2>c_1+\beta$ and $x\in[-l, \infty)$, we have
	\begin{eqnarray}\label{23}
		&V_{(c_1, c_2)}(x)=\left \{
		\begin{array}{ll}
			(c_2-c_1-\beta)\frac{\vartheta^{(q+m, q)}(x, -l)}{\vartheta^{(q+m, q)}(c_2, -l) -\vartheta^{(q+m, q)}(c_1, -l)},    \ \ \  & \mbox{for} \ -l \leq x\leq c_2, \\
			x-c_1-\beta+\frac{(c_2-c_1-\beta)\vartheta^{(q+m, q)}(c_1, -l)}{\vartheta^{(q+m, q)}(c_2, -l) -\vartheta^{(q+m, q)}(c_1, -l)},      & \mbox{for}\ x> c_2,
		\end{array} \right.
	\end{eqnarray}
where $\vartheta^{(q+m, q)}(\cdot, -l)$ is defined in \eqref{9}.
\end{proposition}
\begin{proof}
For $x\in[-l, c_2)$, applying the strong Markov property together with the fact that no dividends are paid out until the surplus process $R$ exceeds the level $c_2$, from \eqref{11} it follows that
	\begin{eqnarray}\label{24}
		&&V_{(c_1, c_2)}(x)=E_x[e^{-q k_{c_2}^+}\mathbf{1}\{k_{c_2}^+<T\}]V_{(c_1, c_2)}(c_2)
		=\frac{\vartheta^{(q+m, q)}(x,-l)}{\vartheta^{(q+m, q)}(c_2, -l)}V_{(c_1, c_2)}(c_2).
	\end{eqnarray}
We now examine the value of $V_{(c_1, c_2)}$ at $x=c_1$ and $x=c_2$. When $X$ has paths of BV, for $x=c_1$, by \eqref{24}, we have
	\begin{eqnarray}\label{25}
		V_{(c_1, c_2)}(c_1)=\frac{\vartheta^{(q+m, q)}(c_1,-l)}{\vartheta^{(q+m, q)}(c_2, -l)}V_{(c_1, c_2)}(c_2),
	\end{eqnarray}
as well as for $x=c_2$, a dividend of amount $c_2-c_1$ is paid immediately and incur a transaction cost of size $\beta$, and then re-applying the strong Markov property, we get
	\begin{eqnarray}\label{26}
		V_{(c_1, c_2)}(c_2)=c_2-c_1-\beta+V_{(c_1, c_2)}(c_1).
	\end{eqnarray}
By solving a system of equations given in Eqs. \eqref{25} and \eqref{26}, we obtain
	\begin{eqnarray}
		&&V_{(c_1, c_2)}(c_1)=\frac{(c_2-c_1-\beta)\vartheta^{(q+m, q)}(c_1, -l)}{\vartheta^{(q+m, q)}(c_2, -l) -\vartheta^{(q+m, q)}(c_1, -l)},\label{vc1c2c1}\\
		&&V_{(c_1, c_2)}(c_2)=\frac{(c_2-c_1-\beta)\vartheta^{(q+m, q)}(c_2, -l)}{\vartheta^{(q+m, q)}(c_2, -l) -\vartheta^{(q+m, q)}(c_1, -l)}.\label{vc1c2c2}
	\end{eqnarray}
For the UBV case, the approximation approach proposed by Loeffen, Renaud, and Zhou\cite{Loeffen2014} yields a result identical to that in the BV case.
Further plugging \eqref{vc1c2c2} and \eqref{vc1c2c1} into \eqref{24} and \eqref{26}, respectively, we get the expression of $V_{(c_1, c_2)}(x)$ for $-l\leq x\leq c_2$ in \eqref{23}.

For $x> c_2$, since the surplus process $R$ immediately reduces to $c_1$ whenever it exceeds level $c_2$, we have
$V_{(c_1, c_2)}(x)=x-c_1-\beta+V_{(c_1, c_2)}(c_1)$, and then,
by \eqref{vc1c2c1} we obtain the form of $V_{(c_1, c_2)}(x)$ for $x> c_2$ in \eqref{23}.
\end{proof}
Based on the form of $V_{(c_1, c_2)}(x)$ in \eqref{23}, the optimal points $(c_1, c_2)$ should minimize the function below
\begin{eqnarray*}
	&&H(c_1, c_2)=\frac{\vartheta^{(q+m, q)}(c_2, -l) -\vartheta^{(q+m, q)}(c_1, -l)}{c_2-c_1-\beta},
\end{eqnarray*}
where the domain of $H$ is given by $dom(H)=\{(c_1, c_2): c_1\geq 0, c_2>c_1+\beta\}$. Let $C^*$ be the set of $(c_1, c_2)$ from $dom(H)$ that minimizes function $H$, namely $$C^*=\{(c_1^*, c_2^*)\in dom(H): H(c_1^*, c_2^*)=\inf_{(c_1, c_2)\in dom(H)}H(c_1, c_2)\}.$$
Now, we show that $C^*$ is non-empty, and provide necessary conditions for $(c_1^*, c_2^*)\in C^*$, which will play a central role in numerically finding the optimal parameters $c_1^*$ and $c_2^*$.
\begin{proposition}\label{prop 4.2}
Suppose $W^{(q)}(\cdot)\in \mathbb{C}^1((0, \infty))$. The set $C^*$ is non-empty, and for each ${(c_1^*, c_2^*)}\in C^*$, one of the following mutually exclusive cases applies:
\begin{enumerate}[label=(\roman*)]
\item $c_1^*, c_2^*\in (0, \infty)\setminus \{b\}$, $c_2^*>c_1^* +\beta$ and $\vartheta^{(q+m, q)}{'}(c_2^*, -l)
 =\vartheta^{(q+m, q)}{'}(c_1^*, -l)=H(c_1^*, c_2^*)$;
\item $c_1^*=b$, $c_2^*\in (b+\beta, \infty)$, and $\vartheta^{(q+m, q)}{'}(c_2^*, -l) = H(b, c_2^*)$;
\item $c_1^*=0$, $c_2^*\in (\beta, \infty)\setminus \{b\}$, and
$\vartheta^{(q+m, q)}{'}(c_2^*, -l) = H(0, c_2^*)$;
\item $c_2^*=b$, $c_1^*\in(0, b-\beta)$ with $b>\beta$, and
$\vartheta^{(q+m, q)}{'}(c_1^*, -l)=H(c_1^*, b)$;
 \item $c_1^* = 0$, $c_2^*=b$ with $b>\beta$.
\end{enumerate}
\end{proposition}
\begin{proof}
We start by showing that $C^*$ is non-empty through the following three cases.
The first case: as $c_1\uparrow \infty$, the function $H$ can not attain its minimum, as proved below. Recall that $W^{(q)}$ is positive and strictly increasing.
By \eqref{1wqde}, we have
\begin{align}\label{wws}
&w^{(q)}(c_2; -l)-w^{(q)}(c_1; -l)\geq W^{(q)}(c_2+l)-W^{(q)}(c_1+l).
\end{align}
Then, by \eqref{9} and \eqref{wws}, using the mean value theorem, we obtain
	\begin{align}\label{sdf}
		&H(c_1, c_2)\\
&=\frac{w^{(q)}(c_2;-l)-w^{(q)}(c_1;-l)
+m\int_0^l \big(w^{(q)}(c_2;-l+y)-w^{(q)}(c_1;-l+y)\big)W^{(q+m)}(y)dy}
{c_2-c_1-\beta}\nonumber\\
&\geq \frac{W^{(q)}(c_2+l)-W^{(q)}(c_1+l)
+m\int_0^l \big(W^{(q)}(c_2+l-y)-W^{(q)}(c_1+l-y)\big)W^{(q+m)}(y)\mathrm{d}y}{c_2-c_1}\nonumber\\
& \ \ \ \cdot \frac{c_2-c_1}{c_2-c_1-\beta} \nonumber \\
&> \min_{x\in[c_1, c_2+l]}
W^{(q)}{'}(x)\big(1+m\int_0^l W^{(q+m)}(y)\mathrm{d}y\big)
\geq \min_{x\in[c_1, \infty)}
W^{(q)}{'}(x)\big(1+m\int_0^l W^{(q+m)}(y)\mathrm{d}y\big).\nonumber
	\end{align}
	Because of $\lim_{x\rightarrow \infty}W^{(q)}{'}(x)=\infty$, we have $H(c_1, c_2)\stackrel{c_1\uparrow\infty}{\longrightarrow}\infty$, this implies that the infimum of $H$ is not attained as $c_1\uparrow\infty$.
So, there exists $C_1$ such that $c_1\leq C_1$ and $\inf_{(c_1, c_2)\in dom(H)}$ $H(c_1, c_2)=\inf_{(c_1, c_2)\in dom(H)\wedge (c_1\leq C_1)}H(c_1, c_2)$.
The second case: as $c_2\uparrow \infty$, the minimization of $H$ is not attained, which we demonstrate as follows. By \eqref{9} and \eqref{wws}, using the increasing property of $W^{(q)}$, we have
	\begin{align*}
		& \inf_{c_1\in[0, C_1]}H(c_1, c_2)
\geq \inf_{c_1\in[0, C_1]}\frac{1}{c_2-c_1-\beta} \\
& \cdot \Big(W^{(q)}(c_2+l)-W^{(q)}(c_1+l)
+m\int_0^l \big(W^{(q)}(c_2+l-y)-W^{(q)}(c_1+l-y)\big)W^{(q+m)}(y)\mathrm{d}y\Big)\\
&\geq \frac{W^{(q)}(c_2+l)+m\int_0^l W^{(q)}(c_2+l-y)W^{(q+m)}(y)\mathrm{d}y}{c_2-\beta}\\
 & \ \ \ -\frac{W^{(q)}(C_1+l)}{c_2-C_1-\beta}
 \cdot \big(1+m\int_0^l W^{(q+m)}(y)\mathrm{d}y\big).
	\end{align*}
Then, by L'H\^{o}pital's rule, we have
\begin{align*}
&\lim_{c_2\rightarrow \infty} \inf_{c_1\in[0, C_1]}H(c_1, c_2) = \lim\limits_{c_2\rightarrow\infty} \big(W^{(q)}{'}(c_2+l)+m\int_0^l W^{(q)}{'}(c_2+l-y)W^{(q+m)}(y)\mathrm{d}y\big) = \infty,
 \end{align*}
and it is necessary to restrict $c_2$ to be bounded. The third case: when $(c_1, c_2)$ converges to the line $c_2=c_1+\beta$, the function $H$ does not attain its minimum, the proof is as follows. Since $c_2>c_1+\beta$, by \eqref{sdf}, using the mean value theorem, we have
	\begin{eqnarray*}
		H(c_1, c_2)\geq \min_{x\in[c_1, c_2+l]}
W^{(q)}{'}(x)\big(1+m\int_0^l W^{(q+m)}(y)\mathrm{d}y\big)
\frac{\beta}{c_2-c_1-\beta}
		\stackrel{c_2\downarrow c_1+\beta}{\longrightarrow} \infty.
	\end{eqnarray*}
So then, we can rule out the possibility of $c_2$ converging to $c_1+\beta$.
Based on the aforementioned three cases, combined with the continuity of $H$, we can conclude that the set $C^*$ is non-empty.

Next, we proceed to prove the necessary conditions for $(c_1^*, c_2^*)\in C^*$. Since $W^{(q)}(\cdot)\in \mathbb{C}^{1}((0, \infty))$, by Proposition \ref{prop 4.5}, we get $\vartheta^{(q+m, q)}(\cdot, -l)\in \mathbb{C}^{1}((-l, \infty)\setminus \{0, b\})$ for $X$ is of BV.
For $c_1, c_2\in(0, \infty)\setminus \{b\}$, by considering the function $H$ is partially differentiable in $c_1$ and $c_2$, and it attains a local minimum at point $(c_1^*, c_2^*)$, we have $\frac{\partial H(c_1, c_2)}{\partial c_1}\Big|_{(c_1,c_2)=(c_1^*, c_2^*)}=0$ and $\frac{\partial H(c_1, c_2)}{\partial c_2}\Big|_{(c_1, c_2)=(c_1^*, c_2^*)}=0$,
which, by direct computation, is equivalent to the condition $\vartheta^{(q+m, q)}{'}(c_2^*, -l)= \vartheta^{(q+m, q)}{'}(c_1^*, -l)$, where
\begin{eqnarray}\label{27}
		&&\vartheta^{(q+m, q)}{'}(c_2^*, -l)=\frac{\vartheta^{(q+m, q)}(c_2^*, -l) -\vartheta^{(q+m, q)}(c_1^*, -l)}{c_2^*-c_1^*-\beta} =H(c_1^*, c_2^*),
	\end{eqnarray}
thereby confirming that condition (i) is satisfied.
For $c_1^*=b$ and $c_2^*\in(b+\beta, \infty)$, by considering $c_2^*$ minimizes the function
$H(b, c_2)$, we have
$\frac{\mathrm{d} H(b, c_2)}{\mathrm{d} c_2}|_{c_2=c_2^*}=0$, which implies (ii) holds. For $c_1^*=0$ and $c_2^*\in(\beta, \infty)\setminus \{b\}$, from $\frac{\mathrm{d} H(0, c_2)}{\mathrm{d} c_2}|_{c_2=c_2^*}=0$, we can obtain (iii) holds.
For $c_2^*=b$ and $c_1^*\in(0, b-\beta)$ with $b>\beta$, from $\frac{\mathrm{d} H(c_1, b)}{\mathrm{d} c_1}|_{c_1=c_1^*}=0$, we can also deduce (iv) holds.
Excluding the scenarios listed above, the only remaining admissible configuration is $c_1^*=0, c_2^*=b$ with $b > \beta$.
\end{proof}
From Propositions \ref{prop 4.1} and \ref{prop 4.2}, we can obtain the following result.
\begin{proposition}\label{corol 4.3}
 Suppose $W^{(q)}(\cdot)\in \mathbb{C}^1((0,\infty))$. For each $(c_1^*, c_2^*)\in C^*$ with $c_1^*\geq 0$, $c_2^*\in(c_1^* +\beta, \infty)$ excluding the point $c_2^* =b$ when $X$ is of BV, and $x\in[-l, \infty)$, we have
	\begin{eqnarray*}
		&V_{(c_1^*, c_2^*)}(x)=\left \{
		\begin{array}{ll}
			\frac{\vartheta^{(q+m, q)}(x, -l)}{\vartheta^{(q+m, q)}{'}(c_2^*, -l)},    \ \ \  & \mbox{for} \ -l\leq x\leq c_2^*, \\
			x-c_2^*+\frac{\vartheta^{(q+m, q)}(c_2^*, -l)}{\vartheta^{(q+m, q)}{'}(c_2^*, -l)},      & \mbox{for} \ x> c_2^*.
		\end{array} \right.
	\end{eqnarray*}
\end{proposition}
\begin{proof}
	By \eqref{27}, from Proposition \ref{prop 4.1} it follows that, for $-l\leq x\leq c_2^*$,
	\begin{eqnarray*}
		V_{(c_1^*, c_2^*)}(x)=\frac{c_2^*-c_1^*-\beta}{\vartheta^{(q+m,q)}(c_2^*, -l)-\vartheta^{(q+m, q)}(c_1^*, -l)}\vartheta^{(q+m, q)}(x, -l)=\frac{\vartheta^{(q+m, q)}(x, -l)}{\vartheta^{(q+m, q)}{'}(c_2^*, -l)}.
	\end{eqnarray*}
	For $x> c_2^*$, by \eqref{27}, we have
	\begin{align*}
		V_{(c_1^*, c_2^*)}(x)
&=x-c_1^*-\beta+\frac{c_2^*-c_1^*-\beta}{\vartheta^{(q+m,q)}(c_2^*, -l)-\vartheta^{(q+m, q)}(c_1^*, -l)}\vartheta^{(q+m, q)}(c_1^*, -l)\\
&=x-c_1^*-\beta+\frac{\vartheta^{(q+m, q)}(c_1^*, -l)}{\vartheta^{(q+m, q)}{'}(c_2^*, -l)}\\
&=x-c_2^*+\frac{(c_2^*-c_1^*-\beta)\vartheta^{(q+m, q)}{'}(c_2^*, -l)+\vartheta^{(q+m, q)}(c_1^*, -l)}{\vartheta^{(q+m, q)}{'}(c_2^*, -l)}\\
&=x-c_2^*+\frac{\vartheta^{(q+m, q)}(c_2^*, -l)}{\vartheta^{(q+m, q)}{'}(c_2^*, -l)}.
	\end{align*}
\end{proof}
Note that $V_{(c_1^*, c_2^*)}$ is an increasing function. In fact, $V_{c_2^*}:=V_{(c_1^*, c_2^*)}$ is the expected discounted dividend function associated with the barrier strategy at level $c_2^*$ for the de Finetti's dividend problem with exponential Parisian ruin and ultimate bankrupt barrier.

\begin{remark}\label{rem:limiting}
In certain limiting parameter regimes, the optimal dividend strategy $\pi_{(c_1^*,c_2^*)}$ and its value function $V_{(c_1^*,c_2^*)}(x)$ reduce to
those of several well-studied dividend problems.
(i)As $l\to\infty$, we have $T^{(c_1^*, c_2^*)}=\tau_p^{(c_1^*,c_2^*)}$, and $V_{(c_1^*,c_2^*)}$ can be expressed in terms of $F^{(q+m,q)}$, since $\lim_{l\to \infty}\vartheta^{(q+m,q)}(x, -l) = F^{(q+m,q)}(x)$ by Corollary~\ref{cor00}.
If in addition $\delta\to 0$, then problem \eqref{optpro} reduces to SNLP dividend problem with Parisian exponential delay and transaction costs, and $V_{(c_1^*,c_2^*)}$ formally coincides with that in Theorem 4.1 of Renaud\cite{Renaud2024}, since $\lim_{l\to\infty, \delta\to 0}\vartheta^{(q+m,q)}(x,-l)=H^{(q+m,q)}(x)$ by Corollary~\ref{cor22}. \\
(ii) As  $l\to 0$ (or $m\to\infty$), we have $T^{(c_1^*, c_2^*)} =k_{0}^{(c_1^*,c_2^*)}$ (classic ruin), and $V_{(c_1^*,c_2^*)}$ can be expressed in terms of $w^{(q)}(x;0)$, since $\lim_{l\to0}\vartheta^{(q+m,q)}(x,-l)=w^{(q)}(x;0)$ by Corollary \ref{cor010}. If in addition $\delta\to 0$, then problem \eqref{optpro} reduces to SNLP dividend problem with classic ruin and transaction costs, and $V_{(c_1^*,c_2^*)}$ coincides with that in Corollary 4 of Loeffen\cite{Loeffen2009}, since $\lim_{l\to0, \delta\to 0}\vartheta^{(q+m,q)}(x,-l)=W^{(q)}(x)$ by Corollary~\ref{cor22}. Moreover, when $\delta, \beta\to 0$, problem \eqref{optpro} further reduces to SNLP dividend problem with classic ruin and without transaction costs, and $V_{(c_1^*,c_2^*)}$ matches the one in Proposition 1 of Loeffen\cite{Loeffen2008}, in particular, $(c_1^*,c_2^*)$ collapses to a single optimal barrier $b^*$.
\end{remark}

\subsection{Optimality of impulse strategy $\pi_{(c_1^*, c_2^*)}$}
The smoothness of the Parisian refracted $(q, m)$-scale function $\vartheta^{(q+m,q)}(\cdot, -l)$ on $[-l, \infty)$ is established in Proposition \ref{prop 4.5}, which, according to Proposition \ref{prop 4.1}, ensures the smoothness of $V_{(c_1^*, c_2^*)}(\cdot)$ on $[-l, \infty)$ under the ${(c_1^*, c_2^*)}$ policy. This is required for the result discussed below, which uses standard Markovian arguments to demonstrate that $V_{(c_1^*, c_2^*)}$ complies with the verification lemma provided in Lemma \ref{lem 3.4}, thereby verifying the optimality of strategy $\pi_{(c_1^*, c_2^*)}$.

\begin{lemma}\label{lem 4.6}
Suppose $\vartheta^{(q+m,q)}(\cdot, -l)$ is sufficiently smooth on $[-l, \infty)$. Then, the function $V_{(c_1^{*}, c_2^{*})}(x)$ satisfies the HJB inequalities \eqref{21}-\eqref{22}, and therefore, $V_{*}=V_{(c_1^{*}, c_2^{*})}$ a.e., and $\pi_*=\pi_{(c_1^*, c_2^*)}$.
\end{lemma}

\begin{proof}
We first prove that $V_{(c_1^{*}, c_2^{*})}$ satisfies the HJB inequality \eqref{21}.
For $-l\leq x<c_2^{*}$, to verify that $(\mathcal{A}-q -m\mathbf{1}\{x<0\})V_{(c_1^*, c_2^*)}(x)=0$, it is sufficient, by Proposition \ref{prop 4.1}, to show that the process $\{e^{-q(t\wedge \tau) - L^{(c_1^*, c_2^*)}(t\wedge \tau)} \vartheta^{(q+m,q)}(U^{(c_1^*, c_2^*)}_{t\wedge \tau}, -l)\}_{t\geq0}$ is a $P_x$-martingale, where $\tau:=T^{(c_1^*, c_2^*)} \wedge \hat{k}_{c_2^*}^+$ with $\hat{k}_{c_2^*}^+$ denoting the first time that $U^{(c_1^*,c_2^*)}$ exceeds the level $c_2^*$, and $L^{(c_1^*,c_2^*)}(t):= \int_0^t m \mathbf{1}\{U_s^{(c_1^*, c_2^*)} <0\} \mathrm{d} s$.
Note that no dividends are paid before time $t\wedge \tau$, i.e. $\tau= T \wedge k_{c_2^*}^+$, $U^{(c_1^*, c_2^*)}_{t\wedge \tau}=R_{t\wedge \tau}$ and $L^{(c_1^*,c_2^*)}(t)= L(t):=\int_0^t m \mathbf{1}\{R_s <0\} \mathrm{d} s$ for all $t\geq 0$. Let $\tilde{\tau}:=(\tau - t)^+ = \max\{0, \tau - t\}$ and $\tilde{R}_{\tilde{\tau}}:= R_{t +\tilde{\tau}}$. By strong Markov property and \eqref{11}, together with the fact $\vartheta^{(q+m,q)}(R_\tau, -l)=\vartheta^{(q+m,q)}(c_2^*, -l)\mathbf{1}\{k_{c_2^*}^+<T\}$, we have,
	\begin{align*}
		&E_x[e^{-q \tau - L(\tau)}\vartheta^{(q+m,q)}(R_\tau, -l)\mid \mathcal{F}_t]\\
&= \mathbf{1}\{t\leq \tau\}e^{-q t - L(t)}E_{x}[e^{-q (\tau-t) - L(\tau-t)}\vartheta^{(q+m,q)}(R_{\tau}, -l)\mid \mathcal{F}_t] \\ & \ \ \ +\mathbf{1}\{\tau<t\}e^{-q \tau - L(\tau)}\vartheta^{(q+m,q)}(R_\tau, -l)\\
&=\mathbf{1}\{t\leq \tau\}e^{-q t - L(t)}E_{R_t}[e^{-q \tilde{\tau} - L(\tilde{\tau})}\vartheta^{(q+m,q)}(\tilde{R}_{\tilde{\tau}}, -l)]
+\mathbf{1}\{\tau<t\}e^{-q \tau - L(\tau)}\vartheta^{(q+m,q)}(R_\tau, -l)
		\\ & =\mathbf{1}\{t\leq \tau\}e^{-q t - L(t)}\vartheta^{(q+m,q)}(R_t, -l)+\mathbf{1}\{\tau<t\}e^{-q \tau - L(\tau)}\vartheta^{(q+m,q)}(R_\tau, -l)
\\ & =e^{-q (t\wedge \tau) - L(t\wedge \tau)}\vartheta^{(q+m,q)}(R_{t\wedge \tau}, -l),
	\end{align*}
which implies the process $\{e^{-q(t\wedge \tau) - L^{(c_1^*, c_2^*)}(t\wedge\tau)} \vartheta^{(q+m,q)}(U^{(c_1^*, c_2^*)}_{t\wedge \tau}, -l)\}_{t\geq0}$ is a $P_x$-martingale. Alternatively, a more general verification of \eqref{21} can be established as follows.
As derived from Proposition \ref{prop 4.5}, the derivative of $\vartheta^{(q+m,q)}(\cdot, -l)$ fails to exist at $0$ or $b$ when $X$ is of BV, and the second left derivative of $\vartheta^{(q+m,q)}(\cdot, -l)$ at $c_2^*$ does not necessarily equal zero when $X$ is of UBV. Accordingly, $(\mathcal{A}-q-m\mathbf{1}\{x<0\})V_{(c_1^{*}, c_2^{*})}(x)$ is not well-defined.
Therefore, we moderate our claim to demonstrate that the following result holds for any $t\in[0, T]$,
	\begin{eqnarray}\label{28}
		\int_0^t e^{-q s -L^{(c_1^*, c_2^*)}(s)}(\mathcal{A}-q-m\mathbf{1}\{\tilde{U}_s^{(c_1^*, c_2^*)}<0\})V_{(c_1^*, c_2^*)}(\tilde{U}_s^{(c_1^*, c_2^*)})\mathrm{d}s\leq0, \ \ \ a.s.,
	\end{eqnarray}
where the semi-martingale $\tilde{U}^{(c_1^*, c_2^*)}$ is the right-continuous modification of $U^{(c_1^*, c_2^*)}$. The proof of \eqref{28} involves the application of the occupation formula for the semi-martingale local time referenced in Corollary 1, p.219, Protter\cite{Protter2005}.
Specifically,
when $X$ has paths of UBV, this is already established by Lemma 6 of Loeffen\cite{Loeffen2009}.
When $X$ has paths of BV, since it is a quadratic pure jump semi-martingale, as found in Theorem 26, p.71 of Protter\cite{Protter2005}, \eqref{28} automatically holds.

The next step is to show that $V_{(c_1^{*}, c_2^{*})}$ satisfies the HJB inequality \eqref{22}. For $x\geq y> c_2^*$, by \eqref{23}, one has $V_{(c_1^*, c_2^*)}(x)-V_{(c_1^*, c_2^*)}(y)=x-y\geq x-y-\beta \mathbf{1}\{x>y\}$.
For $c_2^*\geq x \geq y$, since $(c_1^*, c_2^*)\in C^*$ minimizes $H(c_1, c_2)$ for $(c_1, c_2)\in dom(H)$, by \eqref{23}, one has
	\begin{align*}
		V_{(c_1^*, c_2^*)}(x)-V_{(c_1^*, c_2^*)}(y)
&
= \frac{(c_2^*-c_1^*-\beta)\big(\vartheta^{(q+m, q)}(x, -l)-\vartheta^{(q+m, q)}(y, -l)\big)}{\vartheta^{(q+m, q)}(c_2^*, -l)-\vartheta^{(q+m, q)}(c_1^*, -l)}\nonumber\\ & \geq \frac{(x-y-\beta\mathbf{1}\{x>y+\beta\})\big(\vartheta^{(q+m, q)}(x, -l)-\vartheta^{(q+m, q)}(y, -l)\big)}{\vartheta^{(q+m, q)}(x, -l)-\vartheta^{(q+m, q)}(y, -l)}\nonumber\\
& = x-y-\beta\mathbf{1}\{x>y+\beta\}\geq x-y-\beta\mathbf{1}\{x>y\}.
	\end{align*}
For $x> c_2^*\geq y$, by using \eqref{23} once again, one gets
	\begin{align*}
		V_{(c_1^*, c_2^*)}(x)-V_{(c_1^*, c_2^*)}(y)
&
=x-c_2^*+ \frac{(c_2^*-c_1^*-\beta)\big(\vartheta^{(q+m, q)}(c_2^*, -l)-\vartheta^{(q+m, q)}(y, -l)\big)}{\vartheta^{(q+m, q)}(c_2^*, -l)-\vartheta^{(q+m, q)}(c_1^*, -l)}.
	\end{align*}
Then, since $H(y, c_2^*)\geq H(c_1^*, c_2^*)$ for $c_2^*>y+\beta$, the above equation satisfies
	\begin{align*}
		&V_{(c_1^*, c_2^*)}(x)-V_{(c_1^*, c_2^*)}(y)\\
&\geq x-c_2^*+ \frac{(c_2^*-y-\beta\mathbf{1}\{c_2^*>y+\beta\})\big(\vartheta^{(q+m, q)}(c_2^*, -l)-\vartheta^{(q+m, q)}(y, -l)\big)}{\vartheta^{(q+m, q)}(c_2^*, -l)-\vartheta^{(q+m, q)}(y, -l)}\nonumber\\
&
=x-y-\beta\mathbf{1}\{c_2^*> y+\beta\}\geq x-y-\beta\mathbf{1}\{x > y+\beta\}\geq x-y-\beta\mathbf{1}\{x> y\}.
	\end{align*}

Thus, $V_{(c_1^{*}, c_2^{*})}$ satisfies both \eqref{21} and \eqref{22}. By Lemma \ref{lem 3.4}, we have $V_{*}=V_{(c_1^{*}, c_2^{*})}$ for a.e. $x\in[-l, \infty)$, and $\pi_{(c_1^*, c_2^*)}$ is the optimal dividend strategy for the impulse control problem \eqref{optpro}.
\end{proof}
The following is the main finding of this section. It establishes a sufficient condition for the optimality of the impulse control $\pi_{(c_1^*, c_2^*)}$, and provides a monotonicity-based criterion for identifying the admissible region of $(c_1^*, c_2^*)$, as demonstrated through numerical results in Section \ref{sec 5}.
\begin{theorem}\label{theorem 4.7}
	Suppose $\vartheta^{(q+m,q)}(\cdot, -l)$ is sufficiently smooth on $[-l, \infty)$, and that there exists ${(c_1^*, c_2^*)}\in C^{*}$ with $c_1^*\geq 0$, $c_2^*\in(c_1^* +\beta, \infty)$ excluding the point $c_2^* =b$ when $X$ is of BV, such that, for all $x, y\in[c_2^*, \infty)$ and $x\geq y$,
	\begin{eqnarray}\label{29}
		&&\vartheta^{(q+m,q)}{'}(x, -l)\geq \vartheta^{(q+m,q)}{'}(y, -l),
	\end{eqnarray}
where $x, y=b$ are excluded in the BV case. Then, $\pi_{(c_1^*, c_2^*)}$ is the optimal strategy for the impulse control problem \eqref{optpro}.
\end{theorem}
\begin{proof}
	The validity of the HJB inequality \eqref{22} follows from the same argument used in the proof of Lemma \ref{lem 4.6}. It is therefore sufficient to verify that \eqref{21} holds for a.e. $x\in [-l, \infty)$. Recall that by Proposition \ref{corol 4.3} one has $V_{(c_1^{*}, c_2^{*})}=V_{c_2^*}$. For $-l\leq x<c_2^*$, as derived in the proof of Lemma \ref{lem 4.6}, it follows that
\begin{align}\label{h123}
&(\mathcal{A}-q -m\mathbf{1}\{x<0\})V_{c_2^{*}}(x)=0.
\end{align}
For $x>y, c_2^*$, using ideas from Theorem 2 of Loeffen\cite{Loeffen2008} together with \eqref{29}, we get
\begin{align}\label{k123}
&\lim_{y\uparrow x}(\mathcal{A}-q-m\mathbf{1}\{y<0\})(V_{c_2^*}-V_{x})(y)\leq 0,
 \end{align}
and then, we prove \eqref{21} by contradiction. Specifically, suppose there exist $x>c_2^*$ such that $(\mathcal{A}-q-m\mathbf{1}\{x<0\})V_{c_2^{*}}(x)>0$. Then, by \eqref{k123} and the continuity of $(\mathcal{A}-q-m\mathbf{1}\{x<0\})V_{c_2^*}(x)$, we have $\lim_{y\uparrow x} (\mathcal{A}-q-m\mathbf{1}\{y<0\}) V_{x}(y)>0$, which contradicts \eqref{h123}, and hence \eqref{21} holds for $x>c_2^*$.
\end{proof}

\section{Examples}\label{sec 5}

The existence and satisfiability of the optimal impulse control $\pi_{(c_1^*, c_2^*)}$ are showed in Proposition \ref{prop 4.2}. Combined with Theorem \ref{theorem 4.7}, this result enables the development of a method for numerically computing the optimal thresholds. Compared to the de Finetti's problem discussed in Loeffen\cite{Loeffen2008}, the additional complexity involves characterizing the optimal values of parameters $c_1$ and $c_2$, which forms a two-dimensional optimization problem, in contrast to the one-dimensional problem of finding the optimal barrier strategy. In this section, we analyze two examples for which it is able to present the numerical results of the optimal points $(c_1^*, c_2^*)$, including refracted Brownian risk model and refracted Cram\'{e}r-Lundberg model with exponential claims.

\subsection{Refracted Brownian risk model}
Let $X$ and $Y$ be Brownian risk processes, namely,
\begin{eqnarray*}
	&&X_t-X_0=\mu t+\sigma B_t \ \ \mbox{and} \ \ Y_t-Y_0=(\mu-\delta) t+\sigma B_t,
\end{eqnarray*}
where $\mu>0$ is a premium rate, $\sigma>0$ is a constant diffusion volatility and $\{B_t, t\geq 0\}$ represents standard Brownian motion. In this case, the Laplace exponent of $X$ is given by $\psi(\lambda)=\mu \lambda+(1/2) \sigma^2 \lambda^2$, and the positive safety loading condition is defined as $\mu-\delta\geq 0$. The $q$-scale function of $X$ is expressed as
\begin{eqnarray*}
	&&W^{(q)}(x)=\frac{2}{\sigma^2 \rho}(e^{\rho_2 x}-e^{\rho_1 x}),
\end{eqnarray*}
where $\rho_1=\frac{-\sqrt{\mu^2+2q\sigma^2}-\mu}{\sigma^2}$, $\rho_2=\frac{\sqrt{\mu^2+2q\sigma^2}-\mu}{\sigma^2}$ and $\rho=\rho_2-\rho_1=\frac{2\sqrt{\mu^2+2q\sigma^2}}{\sigma^2}$.
The respective parameters for $W^{(q+m)}$ we will denote using superscript $*$, and for process $Y$, we will use superscript $Y$. Furthermore, since the parameters of $W^{(q)}$ and $\mathbb{W}^{(q)}$ satisfy $\frac{\rho_i}{\rho_i-\rho_2^Y}-\frac{\rho_i}{\rho_i-\rho_1^Y}=-\frac{\sigma^2 \rho^Y}{2 \delta}, i=1,2$, the explicit expression for $w^{(q)}$ is as follows
\begin{eqnarray*}
	&&w^{(q)}(x; -l)=\sum_{i=1}^{2}(-1)^i e^{\rho_i l} \sum_{j=1}^{2} (-1)^{j+1} A_{ij} e^{\rho_j^Y x},\ \ \ \mbox{for}\ x\geq b,
\end{eqnarray*}
where $A_{ij}=\frac{4\delta \rho_i e^{(\rho_i-\rho_j^Y)b}}{\sigma^4 \rho \rho^Y (\rho_i-\rho_j^Y)}$. With the formulas $\frac{1}{\rho_i^*-\rho_2}-\frac{1}{\rho_i^*-\rho_1}=\frac{\sigma^2 \rho}{2 m}$ and $\frac{1}{\rho_i-\rho_2^*}-\frac{1}{\rho_i-\rho_1^*}=-\frac{\sigma^2 \rho^*}{2 m}, i=1,2$, we also get
\begin{eqnarray}\label{30}
	&\vartheta^{(q+m, q)}(x, -l)=\left \{
	\begin{array}{ll}
		\sum_{i=1}^{2}(-1)^i (2/\sigma^2 \rho^*)e^{\rho_i^* (x+l)},  &\mbox{for}\ -l\leq x<0, \\
		\sum_{i=1}^{2}(-1)^i e^{\rho_i^* l} \sum_{j=1}^{2} (-1)^{j} A_{ij}^* e^{\rho_j x},   &\mbox{for}\ 0\leq x<b, \\
		\sum_{i=1}^{2}(-1)^i Q_i(l) \sum_{j=1}^{2} (-1)^{j+1} A_{ij} e^{\rho_j^Y x}, &\mbox{for}\ x\geq b,
	\end{array} \right.
\end{eqnarray}
where $A_{ij}^*=\frac{4m}{\sigma^4 \rho\rho^* (\rho_i^*-\rho_j)}$ and $Q_i(l)=\frac{\rho_i-\rho_1^*}{\rho^*}e^{\rho_2^* l}-\frac{\rho_i-\rho_2^*}{\rho^*}e^{\rho_1^* l}$.

We first present the monotonicity of the first derivative of the Parisian refracted $(q, m)$-scale function.
\begin{proposition}\label{prop 5.1}
For any $q, m>0$, $b\geq0$ and $l>0$, there exist a constant vector $\varepsilon=(\varepsilon_1, \varepsilon_2)$ with $\varepsilon_1\leq 0\leq \varepsilon_2$ such that $\vartheta^{(q+m, q)}{'}(\cdot, -l)$ is decreasing on $(-l, \varepsilon_1)\cup(0, \varepsilon_2)$ and increasing on $(\varepsilon_1, 0)\cup(\varepsilon_2, \infty)$.
\end{proposition}

\begin{proof}
	To demonstrate this, we consider the second derivative of $\vartheta^{(q+m, q)}(x, -l)$ with respect to $x$. Based on the form of $\vartheta^{(q+m, q)}(x, -l)$ as given by \eqref{30}, we separately analyze the cases where $-l< x<0$, $0< x<b$ and $x> b$. For $-l< x<0$, one has
$$\vartheta^{(q+m,q)}{''}(x, -l)=W^{(q+m)}{''}(x+l)=(2/ \sigma^2 \rho^*)({\rho_2^*}^2 e^{\rho_2^* (x+l)}-{\rho_1^*}^2 e^{\rho_1^* (x+l)}).$$
Since $W^{(q+m)}{''}(x)$ is clearly increasing and bounded in $x\in(-l, 0)$, and the solution to the equation $W^{(q+m)}{''}(x+l)=0$ is denoted by $\zeta_1:=(1/\rho^*) \ln{({\rho_1^*}^2 / {\rho_2^*}^2 )}-l$, we have $\zeta_1>-l$, and $\vartheta^{(q+m,q)}{'}(\cdot, -l)$ is decreasing on $(-l, \varepsilon_1)$ and increasing on $(\varepsilon_1, 0)$, where $\varepsilon_1:=\zeta_1\wedge 0$.

For $0< x<b$, by \eqref{30}, one has
	\begin{align*}
\vartheta^{(q+m,q)}{''}(x, -l)
&={\rho_2}^2(A_{22}^* e^{\rho_2^* l}-A_{12}^* e^{\rho_1^* l})e^{\rho_2 x}-{\rho_1}^2(A_{21}^* e^{\rho_2^* l}-A_{11}^* e^{\rho_1^* l})e^{\rho_1 x}\\
& =K_2^* e^{\rho_2 x}-K_1^* e^{\rho_1 x},
	\end{align*}
where $K_i^*:={\rho_i}^2(A_{2i}^* e^{\rho_2^* l}-A_{1i}^* e^{\rho_1^* l}), i=1,2$. From $\rho_1, \rho_1^*<0$, $\rho_2, \rho_2^*>0$ and $\rho_1^*<\rho_1<\rho_2<\rho_2^*$ it is easy to see that the constants $K_i^*$ are strictly positive, and then $\vartheta^{(q+m,q)}{''}(\cdot, -l)$ is increasing and bounded in $x\in(0, b)$. The solution to the above equation that equals zero is given by $\zeta_2:=(1/\rho) \log{(K_1^*/ K_2^*)}$. Then, $\vartheta^{(q+m,q)}{'}(\cdot, -l)$ is decreasing on $(0, (\zeta_2\vee0) \wedge b)$ and increasing on $((\zeta_2\vee0)\wedge b, b)$.

For $x> b$, we have
\begin{align*}
\vartheta^{(q+m, q)}{''}(x, -l)
&={\rho_2^Y}^2\big(A_{12} Q_1(l)-A_{22} Q_2(l)\big)e^{\rho_2^Y x}-{\rho_1^Y}^2\big(A_{11} Q_1(l)-A_{21} Q_2(l)\big)e^{\rho_1^Y x}\\
&=K_2 e^{\rho_2^Y x}-K_1 e^{\rho_1^Y x},
\end{align*}
where $K_i:={\rho_i^Y}^2\big(A_{1i} Q_1(l)-A_{2i} Q_2(l)\big), i=1,2$.
	From $\rho_i<\rho_i^Y, (i=1,2)$ it follows that the constant $K_2$ is strictly positive. If $K_1>0$, then $\vartheta^{(q+m, q)}{''}(x, -l)$ is increasing and unbounded in $x>b$. The solution to the above equation that equals zero is given by $\zeta_3:=(1/\rho^Y) \log{(K_1/ K_2)}$. With comparison and analysis, it is proved that $\zeta_3\leq\zeta_2$, and thereby, for $\varepsilon_2:=(\zeta_2\vee0)\wedge (\zeta_3\vee b)$, the function $\vartheta^{(q+m,q)}{'}(\cdot, -l)$ is decreasing on $(0, \varepsilon_2)$ and increasing on $(\varepsilon_2, \infty)$. Otherwise, if $K_1<0$, then $\vartheta^{(q+m, q)}{''}(x, -l)$ is positive for all $x> b$, and hence $\vartheta^{(q+m, q)}{'}(x, -l)$ is increasing on $(b, \infty)$, and $\varepsilon_2$ is redefined as $\varepsilon_2:=(\zeta_2\vee 0)\wedge b$.
\end{proof}

From Theorem \ref{theorem 4.7} and Proposition \ref{prop 5.1}, we have $c_2^*\in[\varepsilon_2, \infty)$, and with Proposition \ref{prop 4.2}, $c_1^*\in[0, \varepsilon_2]$. Accordingly, the two-dimensional optimization problem to determine the values of $c_1^*$ and $c_2^*$ can be addressed by solving two auxiliary one-dimensional optimization problems. The numerical procedure is as follows.

\begin{remark}
For given model parameters, we first pre-compute $\vartheta^{(q+m,q)}(x,-l)$ and its derivative on a fine uniform grid. Proposition~\ref{prop 5.1} yields constants $\varepsilon_1\le 0\le \varepsilon_2$ at which the monotonicity of $\vartheta^{(q+m,q)\prime}(\cdot,-l)$ changes. By Theorem \ref{theorem 4.7} and Proposition \ref{prop 5.1} we have $c_2^*\in[\varepsilon_2, \infty)$, and Proposition~\ref{prop 4.2} guarantees that, for each fixed $c_2>\varepsilon_2$, the function $H(\cdot, c_2)$ admits a unique minimiser $c_1^*(c_2)$ on $[0, \varepsilon_2)$, obtained as the unique intersection of $\vartheta^{(q+m, q)}{'}(\cdot, -l)$ and $H(\cdot, c_2)$ on this interval.
Finally, $c_2^*$ is determined as the unique solution of $\vartheta^{(q+m, q)}{'}(c_2, -l)=\vartheta^{(q+m, q)}{'}(c_1^*(c_2), -l)$ on $[\varepsilon_2, \bar{c}_2]$ for a sufficiently large upper bound $\bar{c}_2$, solved numerically as a one-dimensional root-finding problem.

Alternatively, restricting $c_1\in[0,\varepsilon_2]$ and $c_2\in[\varepsilon_2,\infty)$, the optimal pair $(c_1^*,c_2^*)$ can also be identified as an intersection of the level sets $\vartheta^{(q+m,q)\prime}(c_1,-l)=H(c_1,c_2)$ and $\vartheta^{(q+m,q)\prime}(c_2,-l)=\vartheta^{(q+m,q)\prime}(c_1,-l)$, located numerically via contour plots. If no intersection arises in this region, the optimal pair is selected among the boundary configurations in cases (ii)-(v) of Proposition \ref{prop 4.2}, by comparing the associated values of $H$.
\end{remark}

Based on this, we can directly derive the following result, for more details, see also Section 4 in Loeffen\cite{Loeffen2009}.

\begin{theorem}\label{theorem 5.2}
	For the refracted Brownian risk model, there is a unique $\pi_{(c_1^*, c_2^*)}$ policy which is optimal for the impulse control problem \eqref{optpro}.
\end{theorem}
What follows is a numerical analysis of the refracted Brownian risk model. We fix the linear drift at $\mu=0.5$, the diffusion volatility at $\sigma=0.75$, the discount factor at $q=0.05$ and the transaction cost at $\beta=1$. Other parameters are set to $m=0.05$, $\delta=0.03$, $b=3$ and $l=6$ unless specified otherwise in the figures.

Figs.\ref{fig:1} and \ref{fig:2} depict the parameter sensitivity of the function $\vartheta^{(q+m,q)}(x, -l)$ and its first derivative $\vartheta^{(q+m,q)}{'}(x, -l)$, respectively, focusing on variables including the Parisian exponential delay $m$, refraction parameter $\delta$, threshold level $b$ and ultimate bankrupt barrier $-l$. Fig.\ref{fig:3} shows the numerical results of the optimal points $(c_1^*, c_2^*)$ with respect to four varying parameters: $\beta\in[0,2]$, $\delta\in[0,0.2]$, $b\in[0,6]$ and $l\in[0,6]$, where the pair $(c_1^*, c_2^*)$ is drawn in the same color in each subgraph.

\begin{figure}[htbp]
	\centering
	\includegraphics[scale=0.65]{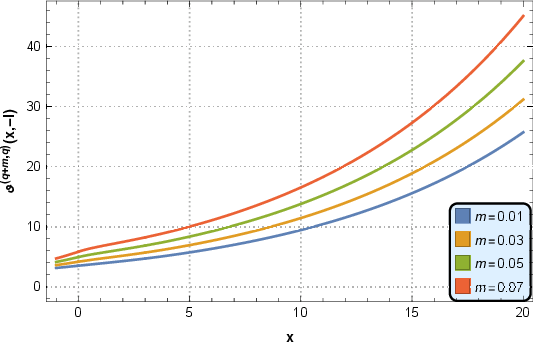}\ \ \ \ \ \ \ \ \ \ \ \ \ \
	\includegraphics[scale=0.65]{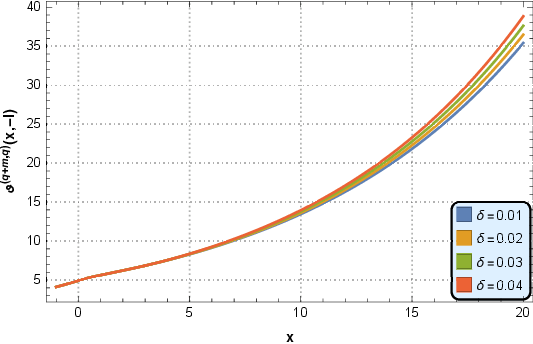}\\
	\includegraphics[scale=0.65]{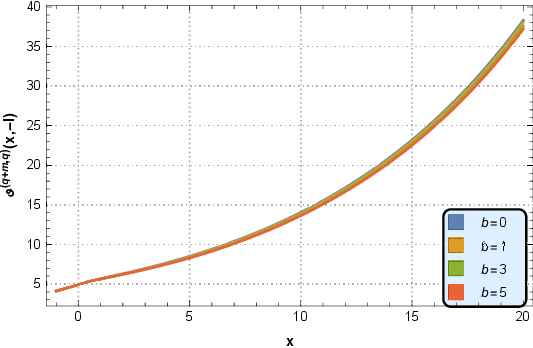}\ \ \ \ \ \ \ \ \ \ \ \ \ \
	\includegraphics[scale=0.65]{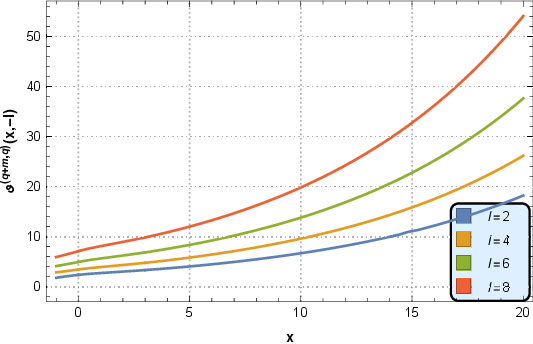}\\
	\caption{Impact of parameters on $\vartheta^{(q+m, q)}(x, -l)$ in a refracted Brownian risk model.}
	\label{fig:1}
\end{figure}
\begin{figure}[htbp]
	\centering
	\includegraphics[scale=0.65]{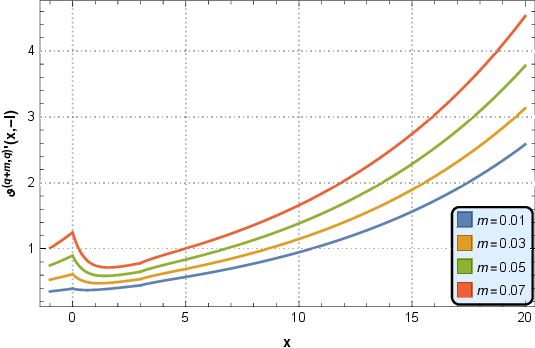}\ \ \ \ \ \ \ \ \ \ \ \ \ \
	\includegraphics[scale=0.65]{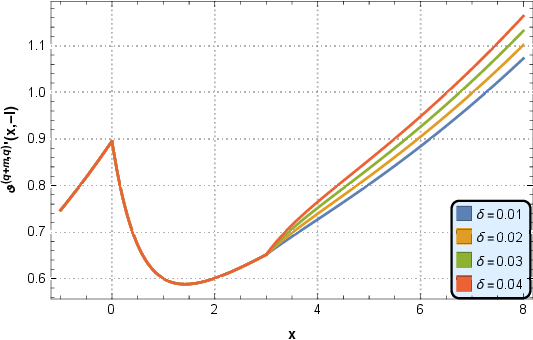}\\
	\includegraphics[scale=0.65]{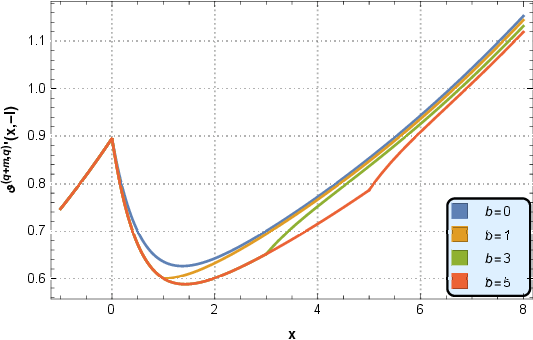}\ \ \ \ \ \ \ \ \ \ \ \ \ \
	\includegraphics[scale=0.65]{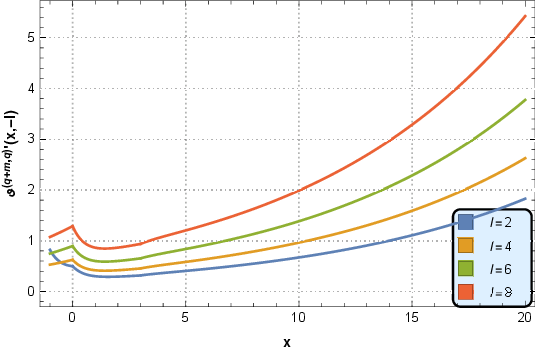}\\
	\caption{Impact of parameters on $\vartheta^{(q+m, q)}{'}(x, -l)$ in a refracted Brownian risk model.}
	\label{fig:2}
\end{figure}

As shown in Fig.\ref{fig:1}, the higher  the initial capital of the insurance company, the higher the value of $\vartheta^{(q+m,q)}(x, -l)$. Analysis of various parameters reveals that larger values of $m$ or $l$ correspondingly increase the value of $\vartheta^{(q+m,q)}(x, -l)$.
Furthermore, for $x\geq b$, $\vartheta^{(q+m,q)}(x, -l)$ increases as $\delta$ increases and decreases as $b$ increases.
Fig.\ref{fig:2} illustrates that, as expected, the derivative $\vartheta^{(q+m, q)}{'}(x, -l)$ decreases in $x\in(0, \varepsilon_2)$ and increases in $x\in(\varepsilon_2, \infty)$. The curve $\vartheta^{(q+m,q)}{'}(x, -l)$ is smooth and exhibits a similar variation tendency to $\vartheta^{(q+m, q)}(x, -l)$ across different values of the parameters $m$, $\delta$, $b$ and $l$.

\begin{figure}[htbp]
	\centering
	\includegraphics[scale=0.65]{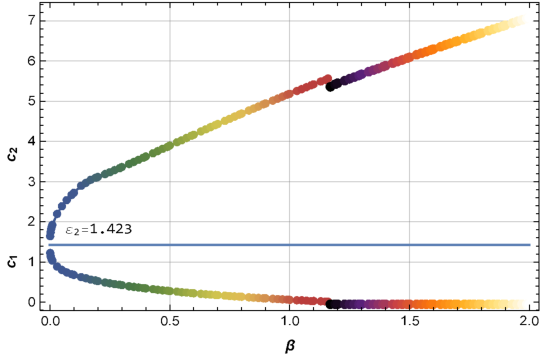}\ \ \ \ \ \ \ \ \ \ \ \ \ \
	\includegraphics[scale=0.65]{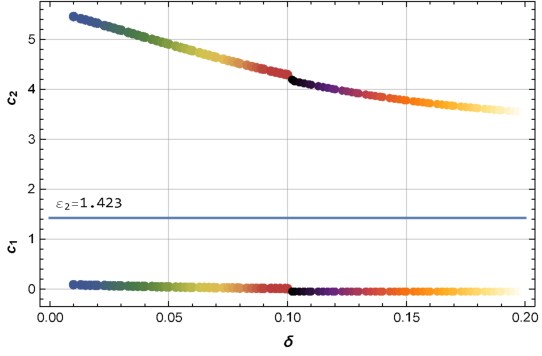}\\
	\includegraphics[scale=0.65]{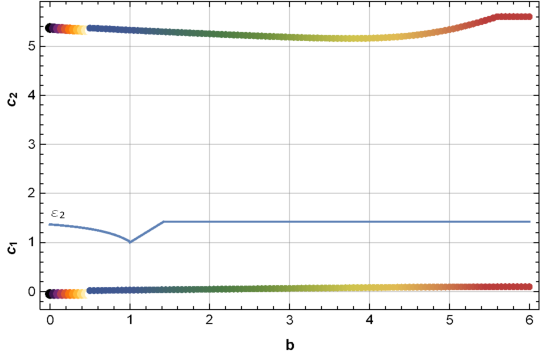}\ \ \ \ \ \ \ \ \ \ \ \ \ \
	\includegraphics[scale=0.65]{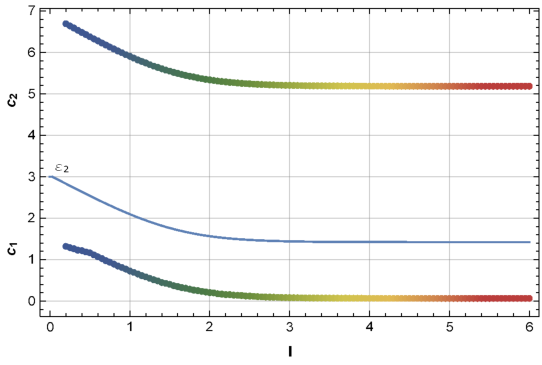}\\
	\caption{Impact of parameters on the optimal pair $(c_1^*, c_2^*)$ in a refracted Brownian risk model.}
	\label{fig:3}
\end{figure}

From Fig.\ref{fig:3}, it is observed that $c_2^*$ is normally above $\varepsilon_2$, and $c_1^*$ is below this level, which is consistent with the theoretical results. The first subgraph illustrates that an increase in $\beta$ leads to a greater distance between $c_1^*$ and $c_2^*$. When $\beta$ is sufficiently large, we have $c_1^*$ approaches zero while $c_2^*$ is sufficiently large, indicating that under excessively high transaction costs, the optimal strategy is to pay as much as possible, and $c_2^*$ needs to be large enough to profit from dividends. The second subgraph shows that increasing $\delta$ causes both $c_1^*$ and $c_2^*$ to decrease. This reflects that a higher $\delta$ reduces the overall drift of the risk process, making it challenging to achieve higher surplus values. The third subgraph indicates that as $b$ increases within the range $[0,6]$, $c_2^*$ initially decreases, then increases, and eventually stabilizes, while $c_1^*$ exhibits a trend of initial decrease followed by an increase. The fourth subgraph reveals that higher values of $l$ result in lower levels of both $c_1^*$ and $c_2^*$.This is mainly due to the fact that lowering the ultimate bankruptcy threshold significantly enhances the operational security of the insurance company in terms of the risk of ruin, thereby allowing for lower levels to achieve higher dividend payments.

\subsection{Refracted Cram\'{e}r-Lundberg model}
Let $X$ and $Y$ be the Cram\'{e}r-Lundberg processes with exponential claims, namely,
\begin{eqnarray*}
	&&X_t-X_0=\mu t-S_t \ \ and \ \ Y_t-Y_0=(\mu-\delta) t-S_t,
\end{eqnarray*}
where $\mu>0$ is a premium rate, $\{S_t=\sum_{i=1}^{N_t} s_i,t\geq0\}$ is a compound Poisson process, $\{N_t,t\geq0\}$ is a homogeneous Poisson process with intensity $\eta>0$, and $\{s_{i},i\geq1\}$ is a sequence of positive independent and exponentially distributed random variables with parameter $\alpha> 0$, $\{s_{i},i\geq1\}$ and $\{N_t,t\geq0\}$ are mutually independent. In this case, the Laplace exponent of $X$ is given by $\psi(\lambda)=\mu \lambda+\eta(\alpha/(\lambda+\alpha)-1)$ for $\lambda>-\alpha$, and the net profit condition is defined as $E[Y_1]=\mu-\delta-\eta/\alpha\geq 0$. The $q$-scale function of $X$ is expressed as
\begin{eqnarray*}
	&&W^{(q)}(x)=\frac{1}{\mu}(G_2 e^{r_2 x}-G_1 e^{r_1 x}),
\end{eqnarray*}
where $r_2=\frac{\sqrt{(\mu \alpha-q-\eta)^2+4\mu q \alpha}-\mu\alpha+q+\eta}{2\mu}$, $r_1=\frac{-\sqrt{(\mu \alpha-q-\eta)^2+4\mu q \alpha}-\mu\alpha+q+\eta}{2\mu}$, $G_i=\frac{\alpha+r_i}{r},\ (i=1,2)$ and $r=r_2-r_1=\frac{\sqrt{(\mu \alpha-q-\eta)^2+4\mu q \alpha}}{\mu}$. Similar to the one above, superscripts $*$ and $Y$ are for $W^{(q+m)}$ and $\mathbb{W}^{(q)}$ respectively.
For $i\in\{1,2\}$ we write $\{1,2\}\backslash i$ for the index in $\{1,2\}$ different from $i$.
From $\frac{r_i}{r_i-r_2^Y}=-\frac{\mu-\delta}{\delta}\frac{r_i-r_1^Y}{r_i+\alpha}$ and $\frac{G_2^Y r_i}{r_i-r_2^Y}-\frac{G_1^Y r_i}{r_i-r_1^Y}=-\frac{\mu-\delta}{\delta},\ (i=1,2)$ it follows that, for $x\geq b$,
\begin{eqnarray*}
	&&w^{(q)}(x; -l)=\frac{1}{\mu r}\sum_{i=1}^{2}(-1)^{i+1} G_i^Y \phi_i(l) e^{r_i^Y x},
\end{eqnarray*}
where $\phi_i(l)=\sum_{j=1}^{2}(-1)^{j+1} (r_j-r_{\{1,2\}\backslash i}^Y) e^{(r_j-r_i^Y)b+r_j l}$. With the formulas $\frac{1}{r_i^*-r_2}=\frac{\mu}{m}\frac{r_i^*-r_1}{r_i^*+\alpha}$, $\frac{G_2}{r_i^*-r_2}-\frac{G_1}{r_i^*-r_1}=\frac{\mu}{m}$ and $\frac{G_2^*}{r_2^*-r_i}-\frac{G_1^*}{r_1^*-r_i}=\frac{\mu}{m},\ (i=1,2)$, we obtain
\begin{eqnarray}\label{31}
	&\vartheta^{(q+m, q)}(x, -l)=\left \{
	\begin{array}{ll}
		(1/\mu)\sum_{i=1}^{2}(-1)^i G_i^*e^{r_i^* (x+l)},    &\mbox{for} \ -l\leq x<0, \\
		(1/\mu r^*)\sum_{i=1}^{2}(-1)^{i+1}G_i \phi_i^*(l)e^{r_i x},    &\mbox{for} \ 0\leq x<b, \\
		(1/\mu r r^* )\sum_{i=1}^{2}(-1)^{i+1}G_i^Y \phi_i^Y(l)e^{r_i^Y x},     &\mbox{for} \ x\geq b,
	\end{array} \right.
\end{eqnarray}
where
\begin{eqnarray*}
&&\phi_i^*(l)=\sum_{j=1}^{2}(-1)^{j+1}(r_j^*-r_{\{1,2\}\backslash i})e^{r_j^* l},\\ &&\phi_i^Y(l)=\sum_{j=1}^{2}(-1)^{j}e^{r_j^*l}\sum_{k=1}^{2}(-1)^{k+1}(r_k-r_{\{1,2\} \backslash i}^Y)(r_j^*-r_{\{1,2\} \backslash k})e^{(r_k-r_i^Y)b}.
\end{eqnarray*}

Initially, we present a proposition that characterizes the monotonicity of the derivative $\vartheta^{(q+m, q)}{'}(\cdot, -l)$. Then, in conjunction with Theorem \ref{theorem 4.7}, we propose a theorem that confirms the uniqueness of the optimal impulse policy $\pi_{(c_1^*, c_2^*)}$.

\begin{proposition}\label{prop 5.3}
For any $q, m>0$, $b\geq0$ and $l>0$, there exist a constant vector $\tilde{\varepsilon}=(\tilde{\varepsilon}_1, \tilde{\varepsilon}_2)$ with $\tilde{\varepsilon}_1\leq 0\leq \tilde{\varepsilon}_2$ such that $\vartheta^{(q+m, q)}{'}(\cdot, -l)$ is decreasing on $(-l, \tilde{\varepsilon}_1)\cup(0, \tilde{\varepsilon}_2) \setminus \{b\}$ and increasing on $(\tilde{\varepsilon}_1, 0)\cup(\tilde{\varepsilon}_2, \infty)\setminus \{b\}$.
\end{proposition}
\begin{proof}
	For $-l< x<0$, one has
	\begin{eqnarray*}
		&&\vartheta^{(q+m,q)}{''}(x, -l)=W^{(q+m)}{''}(x+l)=(1/ \mu)(G_2^* {r_2^*}^2 e^{r_2^* (x+l)}-G_1^* {r_1^*}^2 e^{r_1^* (x+l)}).
	\end{eqnarray*}
	Since $r_1^*<r_1<0<r_2<r_2^*$ and $G_1^*, G_2^*>0$, we have $W^{(q+m)}{''}(x)$ is increasing and bounded in $x\in(-l, 0)$. Denote by $\tilde{\zeta}_1:=\frac{1}{r^*} \log{\frac{G_1^* {r_1^*}^2}{G_2^*{r_2^*}^2 }}-l$ the solution to the equation $W^{(q+m)}{''}(x+l)=0$. Then, $\vartheta^{(q+m,q)}{'}(\cdot, -l)$ is decreasing on $(-l, \tilde{\varepsilon}_1)$ and increasing on $(\tilde{\varepsilon}_1, 0)$, where $\tilde{\varepsilon}_1:=(\tilde{\zeta}_1 \vee -l) \wedge 0$.

For $0< x<b$, by \eqref{31}, one gets
	\begin{eqnarray*}
		&&\vartheta^{(q+m,q)}{''}(x, -l)=(1/\mu r^*)\big(G_1 {r_1}^2 \phi_1^*(l) e^{r_1 x}-G_2 {r_2}^2 \phi_2^*(l) e^{r_2 x}\big).
	\end{eqnarray*}
	Since $G_1, G_2>0$ and $\phi_1^*(l), \phi_2^*(l)>0$, the solution to the above equation that equals zero is
	\begin{eqnarray*}
		&&\tilde{\zeta}_2:=\frac{1}{r} \log{\frac{G_1 {r_1}^2  \phi_1^*(l)}{G_2 {r_2}^2 \phi_2^*(l)}}.
	\end{eqnarray*}
For $x> b$, we deduce that
	\begin{eqnarray*}
		&&\vartheta^{(q+m,q)}{''}(x, -l)=(1/\mu r r^*)\big(G_1^Y {r_1^Y}^2 M_1^Y(l) e^{r_1^Y x}-G_2^Y {r_2^Y}^2 M_2^Y(l) e^{r_2^Y x}\big).
	\end{eqnarray*}
	By $r_1^Y<r_1<0<r_2<r_2^Y$ and $G_1^Y, G_2^Y>0$, using the definition of $\vartheta^{(q+m,q)}(x, -l)$ in \eqref{9} together with the fact that $\lim_{x\rightarrow\infty} w^{(q)}(x; -l)=\infty$, we obtain $M_2^Y(l)<0$. If $M_1^Y(l)<0$, then $\vartheta^{(q+m, q)}{''}(x, -l)$ is increasing and unbounded in $x>b$, and the solution to the above equation that equals zero is
	\begin{eqnarray*}
		&&\tilde{\zeta}_3:=\frac{1}{r^Y} \log{\frac{G_1^Y {r_1^Y}^2 M_1^Y(l)}{G_2^Y {r_2^Y}^2 M_2^Y(l)}}.
	\end{eqnarray*}
From the construction of $\tilde{\zeta}_2$ and $\tilde{\zeta}_3$, it is easy to see that $\tilde{\zeta}_3\leq\tilde{\zeta}_2$, and thereby,  for $\tilde{\varepsilon}_2:=(\tilde{\zeta}_2 \vee 0) \wedge (\tilde{\zeta}_3\vee b)$, the function $\vartheta^{(q+m,q)}{'}(\cdot, -l)$ is decreasing on $(0, \varepsilon_2)\setminus \{b\}$ and increasing on $(\varepsilon_2, \infty)\setminus \{b\}$. Otherwise, if $M_1^Y(l)>0$, then $\vartheta^{(q+m, q)}{''}(x, -l)$ is positive for all $x> b$, and hence $\vartheta^{(q+m, q)}{'}(x, -l)$ is increasing on $(b, \infty)$, and $\tilde{\varepsilon}_2$ is redefined as $\tilde{\varepsilon}_2:=(\tilde{\zeta}_2 \vee 0)\wedge b$.
\end{proof}

\begin{theorem}\label{theorem 5.4}
	For the refracted Cram\'{e}r-Lundberg model, there is a unique $\pi_{(c_1^*, c_2^*)}$ policy which is optimal for the impulse control problem \eqref{optpro}.
\end{theorem}
We fix the premium rate at $\mu=3$, the discount factor at $q=0.05$, the exponential distribution parameter at $\alpha=1$, the Poisson intensity at $\eta=2$ and the transaction cost at $\beta=0.1$. Other parameters are set to $m=0.5$, $\delta=0.15$, $b=6$ and $l=6$ unless specified otherwise in the figures. Figs.\ref{fig:4}-\ref{fig:5} provide a detailed sensitivity analysis of the parameters $m$, $\delta$, $b$ and $l$ on the function $\vartheta^{(q+m,q)}(x, -l)$ and its derivative $\vartheta^{(q+m,q)}{'}(x, -l)$, respectively. In Fig.\ref{fig:6}, the numerical results of the optimal points $(c_1^*, c_2^*)$ are shown with respect to four varying parameters: $\beta\in[0,2]$, $\delta\in[0,0.25]$, $b\in[0,7]$ and $l\in[0,7]$.

\begin{figure}[!ht]
	\centering
	\includegraphics[scale=0.65]{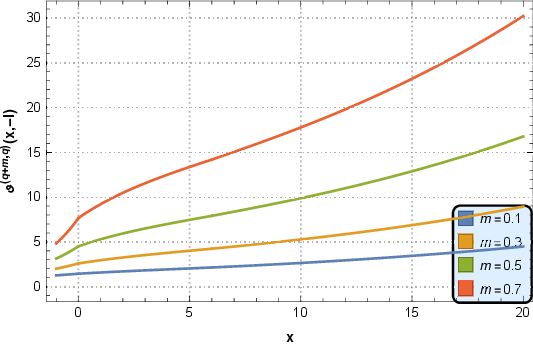}\ \ \ \ \ \ \ \ \ \ \ \ \ \
	\includegraphics[scale=0.65]{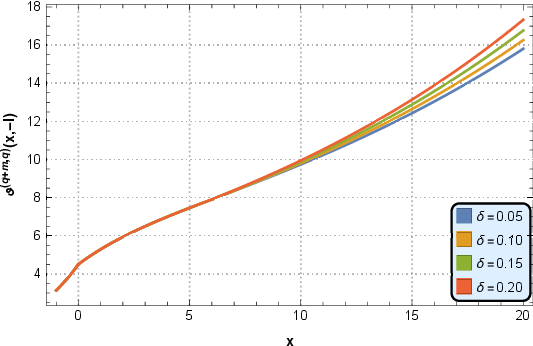}\\
	\includegraphics[scale=0.65]{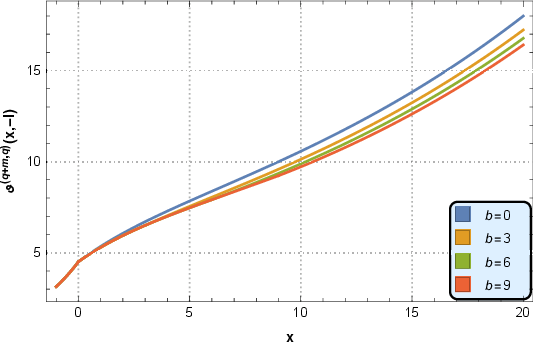}\ \ \ \ \ \ \ \ \ \ \ \ \ \
	\includegraphics[scale=0.65]{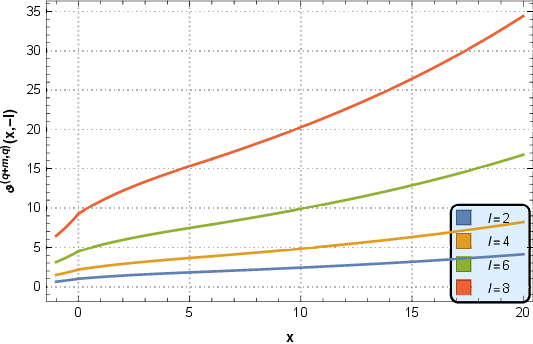}\\
	\caption{Impact of parameters on $\vartheta^{(q+m, q)}(x, -l)$ in a refracted Cram\'{e}r-Lundberg model.}
	\label{fig:4}
\end{figure}
\begin{figure}[!ht]
	\centering
	\includegraphics[scale=0.65]{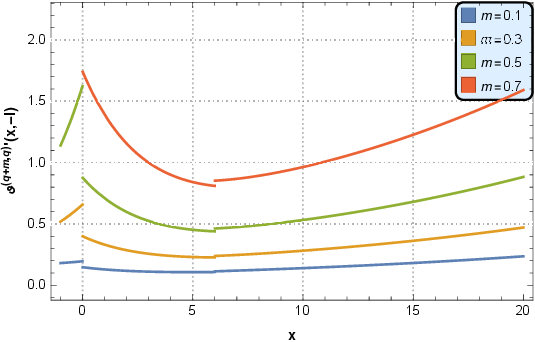}\ \ \ \ \ \ \ \ \ \ \ \ \ \
	\includegraphics[scale=0.65]{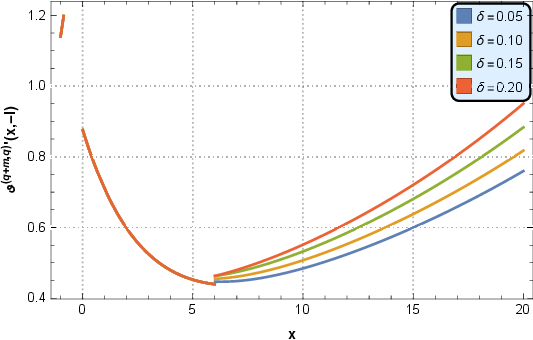}\\
	\includegraphics[scale=0.65]{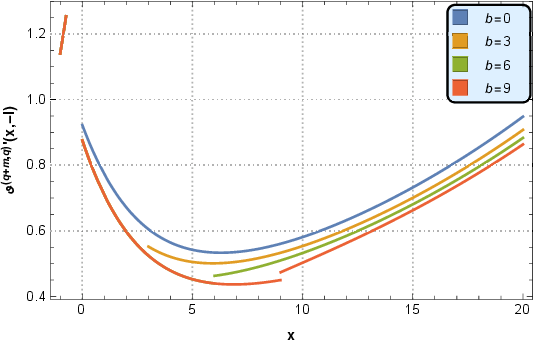}\ \ \ \ \ \ \ \ \ \ \ \ \ \
	\includegraphics[scale=0.65]{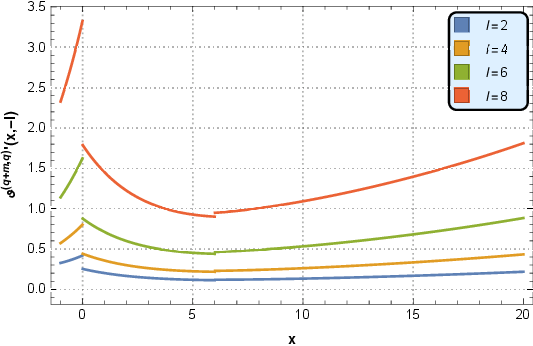}\\
	\caption{Impact of parameters on $\vartheta^{(q+m, q)}{'}(x, -l)$ in a refracted Cram\'{e}r-Lundberg model.}
	\label{fig:5}
\end{figure}
\begin{figure}[!ht]
	\centering
	\includegraphics[scale=0.65]{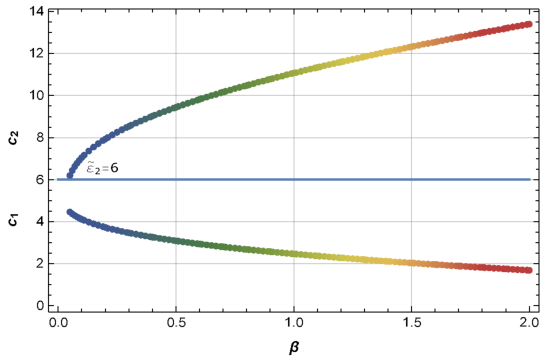}\ \ \ \ \ \ \ \ \ \ \ \ \ \
	\includegraphics[scale=0.65]{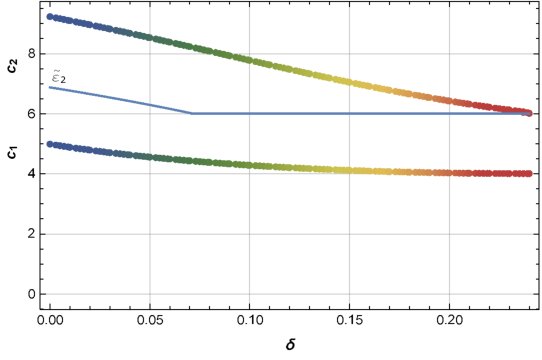}\\
	\includegraphics[scale=0.65]{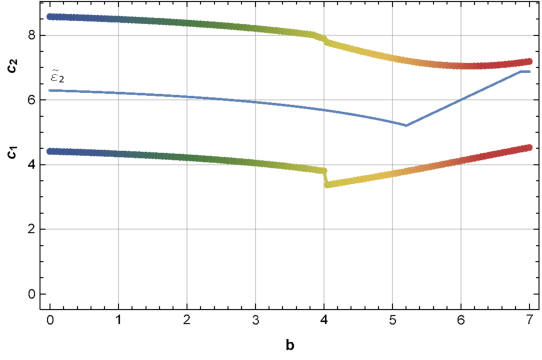}\ \ \ \ \ \ \ \ \ \ \ \ \ \
	\includegraphics[scale=0.65]{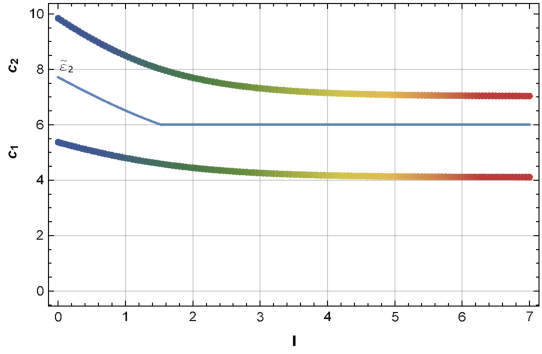}\\
	\caption{Impact of parameters on the optimal pair $(c_1^*, c_2^*)$ in a refracted Cram\'{e}r-Lundberg model.}
	\label{fig:6}
\end{figure}

Fig.\ref{fig:4} shows that $\vartheta^{(q+m, q)}(x, -l)$ monotonically increases with the initial value $x$. Fig.\ref{fig:5} demonstrates that for $x \geq 0$, as shown in Proposition \ref{prop 5.3}, the derivative $\vartheta^{(q+m, q)'}(x, -l)$ initially decreases and then increases as $x$ increases. Notably, $\vartheta^{(q+m, q)}{'}(x, -l)$ is discontinuous at points $0$ and $b$. Both $\vartheta^{(q+m, q)}(x, -l)$ and $\vartheta^{(q+m, q)}{'}(x, -l)$ increase as $m$ or $l$ increases. For $x \geq b$, these functions increase with $\delta$ and decrease with $b$. Fig.\ref{fig:6} indicates that the extreme point $\tilde{\varepsilon}_2$ is usually between $c_1^*$ and $c_2^*$, exceeding $c_1^*$ but not $c_2^*$, which corresponds with theoretical results. The influence of parameters $\beta$, $\delta$, $b$ and $l$ on the optimal impulse policy $(c_1^*, c_2^*)$ is essentially consistent with that observed in the above refracted Brownian risk model.


\begin{thebibliography}{10}

\bibitem{Albrecher2008}
{\sc H.~Albrecher and S.~Thonhauser}, {\em Optimal dividend strategies for a
  risk process under force of interest}, Insurance: Mathematics and Economics,
  43(1) (2008), pp.~134--149.
\url{https://doi.org/10.1016/j.insmatheco.2008.03.012}

\bibitem{Avanzi2008}
{\sc B.~Avanzi and H.~U. Gerber}, {\em Optimal dividends in the dual model with
  diffusion}, ASTIN Bulletin: The Journal of the IAA, 38(2) (2008), pp.~653--667.
\url{https://doi.org/10.2143/AST.38.2.2033357}

  \bibitem{Avram2015}
{\sc F.~Avram, Z.~Palmowski, and M.~R. Pistorius}, {\em On Gerber-Shiu
  functions and optimal dividend distribution for a L\'{e}vy risk process in
  the presence of a penalty function}, The Annals of Applied Probability, 25(4)
  (2015), pp.~1868--1935. \url{https://doi.org/10.1214/14-AAP1038}

\bibitem{Baurdoux20162}
{\sc E.~J. Baurdoux, J.~C. Pardo, J.~L. P\'{e}rez, and J.-F. Renaud}, {\em
  Gerber-Shiu distribution at Parisian ruin for L\'{e}vy insurance risk
  processes}, Journal of Applied probability, 53(2) (2016), pp.~572--584.
\url{https://doi.org/10.1017/jpr.2016.21}

\bibitem{Bayraktar2013}
{\sc E.~Bayraktar, A.~E. Kyprianou, and K.~Yamazaki}, {\em On optimal dividends
  in the dual model}, ASTIN Bulletin: The Journal of the IAA, 43(3) (2013),
  pp.~359--372. \url{https://doi.org/10.1017/asb.2013.17}

  \bibitem{Bouleau1981}
{\sc N.~Bouleau and M.~Yor}, {\em Sur la variation quadratique des temps locaux
  de certaines semimartingales}, Comptes Rendus des s\'{e}ances de l'Acad\'{e}mie des Sciences.S\'{e}rie I.Math\'{e}matique, 292(9) (1981), pp.~491--494.

 \bibitem{Cheung2017}
{\sc E.~C. Cheung and J.~T. Wong}, {\em On the dual risk model with Parisian implementation delays in dividend payments}, European Journal of Operational Research, 257(1) (2017), pp.~159--173. \url{https://doi.org/10.1016/j.ejor.2016.09.018}

\bibitem{Czarna2020}
{\sc I.~Czarna and A.~Kaszubowski}, {\em Optimality of impulse control problem in refracted L\'{e}vy model with Parisian ruin and transaction costs},
  Journal of Optimization Theory and Applications, 185 (2020), pp.~982--1007.
  \url{https://doi.org/10.1007/s10957-020-01682-1}

\bibitem{Czarna2014}
{\sc I.~Czarna and Z.~Palmowski}, {\em Dividend problem with Parisian delay for
  a spectrally negative L\'{e}vy risk process}, Journal of Optimization Theory
  and Applications, 161 (2014), pp.~239--256.
  \url{https://doi.org/10.1007/s10957-013-0283-y}

\bibitem{Czarna2019}
{\sc I.~Czarna, J.-L. P\'{e}rez, T.~Rolski, and K.~Yamazaki}, {\em Fluctuation
  theory for level-dependent L\'{e}vy risk processes}, Stochastic Processes and
  their Applications, 129(12) (2019), pp.~5406--5449.
  \url{https://doi.org/10.1016/j.spa.2019.03.006}

\bibitem{Czarna20162}
{\sc I.~Czarna and J.-F. Renaud}, {\em A note on Parisian ruin with an ultimate
  bankruptcy level for L\'{e}vy insurance risk processes}, Statistics and
  Probability Letters, 113 (2016), pp.~54--61.
  \url{https://doi.org/10.1016/j.spl.2016.02.018}


\bibitem{Dassios2008}
{\sc A.~Dassios and S.~Wu}, {\em Parisian ruin with exponential claims},
  Working paper, LSE, London, 2008.
\url{http://eprints.lse.ac.uk/id/eprint/32033}


\bibitem{De1957}
{\sc B.~De~Finetti}, {\em Su un'impostazion alternativa dell teoria collecttiva
  del rischio}, Transactions of the XVth International Congress of Actuaries, 2
  (1957), pp.~433--443.

  \bibitem{Frostig2020}
{\sc E.~Frostig and A.~Keren-Pinhasik}, {\em Parisian ruin with Erlang delay
  and a lower bankruptcy barrier}, Methodology and Computing in Applied
  Probability, 22 (2020), pp.~101--134.
  \url{https://doi.org/10.1007/s11009-019-09693-w}


  \bibitem{Khoshnevisan2016}
{\sc D.~Khoshnevisan and R. Schilling}, {\em From L\'{e}vy-type processes to
  Parabolic SPDEs}, Advanced Courses in Mathematics, CRM Barcelona,
  Birkh\"{a}user Basel, 2016.

 \bibitem{Kuznetsov2012}
{\sc A.~Kuznetsov, A.~E. Kyprianou, and V.~Rivero}, {\em The theory of scale functions for spectrally negative L\'{e}vy processes}, L\'{e}vy matters II.
  Lecture Notes in Mathematics. 2061,  (2012), pp.~97--186.
  \url{https://doi.org/10.1007/978-3-642-31407-0_2}


  \bibitem{Kyprianou20142}
{\sc A.~E. Kyprianou}, {\em Fluctuations of L\'{e}vy processes with
  applications-introductory lectures}, second, Springer Science and Business
  Media, 2014.


 \bibitem{Kyprianou20101}
{\sc A.~E. Kyprianou and R.~Loeffen}, {\em Refracted L\'{e}vy processess},
  Annales de l'IHP Probabilit\'{e}s et statistiques, 46(1) (2010), pp.~24--44.
  \url{https://doi.org/10.1214/08-AIHP307}


\bibitem{Kyprianou2010}
{\sc A.~E. Kyprianou, V.~Rivero, and R.~Song}, {\em Convexity and smoothness of
  scale functions and de Finetti's control problem}, Journal of Theoretical
  Probability, 23 (2010), pp.~547--564.
\url{https://doi.org/10.1007/s10959-009-0220-z}


\bibitem{landriault2011}
 {\sc D. Landriault, J.-F. Renaud, X. Zhou}, {\em Occupation times of spectrally negative L\'{e}vy processes with applications},
 Stochastic processes and their applications, 121(11) (2011), 2629--2641.  \url{https://doi.org/10.1016/j.spa.2011.07.008}


\bibitem{Landriault2014}
{\sc D.~Landriault, J.-F. Renaud, and X.~Zhou}, {\em An insurance risk model
  with Parisian implementation delays}, Methodology and Computing in Applied
  Probability, 16 (2014), pp.~583--607.
\url{https://doi.org/10.1007/s11009-012-9317-4}


\bibitem{Lkabous2017}
{\sc M.~A. Lkabous, I.~Czarna, and J.-F. Renaud}, {\em Parisian ruin for a
  refracted L\'{e}vy process}, Insurance: Mathematics and Economics, 74 (2017),
  pp.~153--163.
  \url{https://doi.org/10.1016/j.insmatheco.2017.03.005}


\bibitem{Loeffen2013}
{\sc R.~Loeffen, I.~Czarna, and Z.~Palmowski}, {\em Parisian ruin probability
  for spectrally negative L\'{e}vy processes}, Bernoulli, 19(2) (2013),
  pp.~599--609.
  \url{https://doi.org/10.3150/11-BEJ404}


\bibitem{Loeffen2018}
{\sc R.~Loeffen, Z.~Palmowski, and B.~A. Surya}, {\em Discounted penalty
  function at Parisian ruin for L\'{e}vy insurance risk process}, Insurance:
  Mathematics and Economics, 83 (2018), pp.~190--197.
\url{https://doi.org/10.1016/j.insmatheco.2017.10.008}

\bibitem{Loeffen2014}
{\sc R.~Loeffen, J.~F. Renaud, and X.~Zhou}, {\em Occupation times of intervals
  until first passage times for spectrally negative L\'{e}vy processes},
  Stochastic Processes and their Applications, 124(3) (2014), pp.~1408--1435.
  \url{https://doi.org/10.1016/j.spa.2013.11.005}

  \bibitem{Loeffen2008}
{\sc R.~L. Loeffen}, {\em On optimality of the barrier strategy in de Finetti's
  dividend problem for spectrally negative L\'{e}vy processes}, The Annals of
  Applied Probability, 18(5) (2008), pp.~1669--1680.
  \url{https://doi.org/10.1214/07-AAP504}


\bibitem{Loeffen2009}
{\sc R.~L. Loeffen}, {\em An optimal dividends problem with transaction costs
  for spectrally negative L\'{e}vy processes}, Insurance: Mathematics and
  Economics, 45(1) (2009), pp.~41--48.
  \url{https://doi.org/10.1016/j.insmatheco.2009.03.002}


\bibitem{Protter2005}
{\sc P.~Protter}, {\em Stochastic integration and differential equations, 2nd
  ed., version 2.1.}, Springer, Berlin, 2005.

\bibitem{Renaud2019}
{\sc J.-F. Renaud}, {\em De Finetti's control problem with Parisian ruin for
  spectrally negative L\'{e}vy processes}, Risks, 7(3) (2019), pp.73.
  \url{https://doi.org/10.3390/risks7030073}


\bibitem{Renaud2024}
{\sc J.-F. Renaud}, {\em A note on the optimal dividends problem with
  transaction costs in a spectrally negative L\'{e}vy model with parisian
  ruin}, Statistics and Probability Letters, 206 (2024), p.~109978.
  \url{https://doi.org/10.1016/j.spl.2023.109978}

\bibitem{Schilling1998}
{\sc R. Schilling}, {\em Growth and H\"{o}lder conditions for the
  sample paths of Feller processes}, Probability Theory and Related Fields, 112
  (1998), pp.~565--611.
  \url{https://doi.org/10.1007/s004400050201}

\bibitem{Schnurr2012}
{\sc A.~Schnurr}, {\em On the semimartingale nature of Feller processes with
  killing}, Stochastic Processes and their Applications, 122(7) (2012),
  pp.~2758--2780.
  \url{https://doi.org/10.1016/j.spa.2012.04.009}

\bibitem{Thonhauser2011}
{\sc S.~Thonhauser and H.~Albrecher}, {\em Optimal dividend strategies for a
  compound Poisson process under transaction costs and power utility},
  Stochastic Models, 27(1) (2011), pp.~120--140.
  \url{https://doi.org/10.1080/15326349.2011.542734}

\bibitem{Yang2020}
{\sc C.~Yang, K.~P. Sendova, and Z.~Li}, {\em Parisian ruin with a threshold
  dividend strategy under the dual L\'{e}vy risk model}, Insurance: Mathematics
  and Economics, 90 (2020), pp.~135--150.
  \url{https://doi.org/10.1016/j.insmatheco.2019.11.002}
\end{thebibliography}

\section*{Conflict of interest}
All authors declare no conflicts of interest in this paper.

\appendix

\section{Proofs and more}\label{Section appen}

In the following we provide proofs of propositions, corollaries and lemmas for completeness.

\subsection{Proof of Proposition \ref{propequexp}}
\begin{proof}\label{propequexp22}
Let $e_q\sim \exp(q)$ be an exponentially distributed random variable with parameter $q\geq0$, and $L^{\pi}(t)=\int_0^t m \mathbf{1}\{U_s^{\pi} <0\} \mathrm{d} s$.
As shown in Landriault, Renaud and Zhou\cite{landriault2011}, for all $t'\geq t\geq 0$, the probability of Parisian ruin with exponential implementation delays under the controlled process $U^{\pi}$ is
\begin{align}\label{proptaul}
P\big(\tau^{\pi}_{p}>t\big|\mathcal{F}^{U^{\pi}}_{t'}\big)=e^{- L^{\pi}(t)}, \ \mbox{and} \ P\big(\tau_p^{\pi}\in \mathrm{d}t \big| \mathcal{F}^{U^{\pi}}_{t'}) = e^{-L^{\pi}(t)} m \mathbf{1}\{U_t^{\pi}<0\} \mathrm{d} t.
\end{align}
Therefore, we can write
\begin{align}\label{vj1j2}
V_{\pi}(x)
&=E_x\big[\int_0^{\tau_p^{\pi} \wedge k_{-l}^{\pi}} e^{-q t} \mathrm{d}D_{\beta}^{\pi}(t) \big]
=E_{x}\big[\int_{0}^{\infty}\mathbf{1}\{t\leq e_{q}\wedge k_{-l}^{\pi}\wedge\tau_{p}^{\pi}\}
\mathrm{d}D_{\beta}^{\pi}(t)\big] \\
& =E_{x}\big[D_{\beta}^{\pi}\big(e_{q}\wedge k_{-l}^{\pi}\wedge\tau_{p}^{\pi}\big)\big]\nonumber\\
&=E_{x}\big[D_{\beta}^{\pi}\big(e_{q}\wedge k_{-l}^{\pi}\big) \mathbf{1}\{e_{q}\wedge k_{-l}^{\pi}<\tau_{p}^{\pi}\}\big]
+E_{x}\big[D_{\beta}^{\pi}\big(\tau_{p}^{\pi}\big) \mathbf{1}\{\tau_{p}^{\pi}\leq e_{q}\wedge k_{-l}^{\pi}\}\big]=J_{1}+J_{2}.\nonumber
\end{align}
Making use of the fact \eqref{proptaul} and $D_{\beta}^{\pi}(0)=0$, one can find that
\begin{align*}
J_{1}
&=E_{x}\big[e^{-L^{\pi}(e_{q}\wedge k_{-l}^{\pi})} D^{\pi}_{\beta}(e_{q}\wedge k_{-l}^{\pi})\big]
=E_{x}\big[\int_{0}^{e_{q}\wedge k_{-l}^{\pi}}\big(e^{-L^{\pi}(t)}\mathrm{d}D_{\beta}^{\pi}(t)
-e^{-L^{\pi}(t)}D_{\beta}^{\pi}(t) \mathrm{d}L^{\pi}(t)\big)\big]\\
&=E_{x}\big[\int_{0}^{e_{q}\wedge k_{-l}^{\pi}}e^{-L^{\pi}(t)}
\big(\mathrm{d}D_{\beta}^{\pi}(t) - D_{\beta}^{\pi}(t) m\mathbf{1}\{U_t^{\pi}<0\} \mathrm{d}t\big)\big].
\end{align*}
Similarly, by conditioning on $\mathcal{F}_{\tau_{p}^{\pi}}^{U^{\pi}}$, we have
\begin{align*}
J_{2}
&= E_{x}\Big[\int_{0}^{\infty} E_{x}\big[D_{\beta}^{\pi}(t) \mathbf{1}\{t\leq e_{q}\wedge k_{-l}^{\pi}\}\big|\mathcal{F}_{t}^{U^{\pi}}\big] e^{-L^{\pi}(t)} m\mathbf{1}\{U_t^{\pi}<0\}\mathrm{d}t\Big]\\
& = m E_{x}\big[\int_{0}^{\infty}e^{-L^{\pi}(t)}D_{\beta}^{\pi}(t)\mathbf{1}\{U_t^{\pi}<0\}
\mathbf{1}\{t\leq e_{q}\wedge k_{-l}^{\pi}\} \mathrm{d}t\big].
\end{align*}
Plugging the identities above into \eqref{vj1j2} gives
\begin{align*}
V_{\pi}(x)
& =E_{x}\big[\int_{0}^{e_{q}\wedge k_{-l}^{\pi}}e^{-L^{\pi}(t)} \mathrm{d}D_{\beta}^{\pi}(t)\big]
=E_{x}\big[\int_{0}^{k_{-l}^{\pi}}\mathbf{1}\{t\leq e_{q}\}e^{-L^{\pi}(t)}\mathrm{d}D_{\beta}^{\pi}(t)\big]\\
& =E_{x}\big[\int_{0}^{k_{-l}^{\pi}} e^{-q t - \int_0^t m \mathbf{1}\{U_s^{\pi} <0\} \mathrm{d} s} \mathrm{d} D_{\beta}^{\pi}(t)\big].
\end{align*}
\end{proof}

\subsection{Proof of Proposition \ref{prop 3.1}}
\begin{proof}\label{prop3112}
The proof consists of two cases: $c\leq b$ and $c>b$. We focus on the case where $c>b$, as the argument for $c\leq b$ is similar and thus omitted. Note that the exponential delay applies when the surplus $R$ becomes negative, and for $x\in[-l, c)$, we have
\begin{align}\label{exp0xi}
E_x[e^{-q k_{c}^+}\mathbf{1}\{k_{c}^+<T\}]
&= E_x[e^{-q k_{c}^+}\mathbf{1}\{k_{c}^+<\tau_{p} \wedge k_{-l}^-\}]\nonumber\\
&= E_x[e^{-q k_{c}^+ - m \int_0^{k_c^+} \mathbf{1}\{R_s <0\} \mathrm{d} s} \mathbf{1}\{k_c^+<k_{-l}^-\}].
\end{align}

For $b\leq x< c$, considering $\{Y_t, t<\upsilon_{b}^-\}$ and $\{R_t, t<k_{b}^-\}$ have the same distribution, using the strong Markov property of $R$ and the fact that it is skip-free upward, we can obtain
	\begin{align}\label{12}
		 &E_x[e^{-q k_{c}^+}\mathbf{1}\{k_{c}^+<T\}]
=E_x[e^{-q \upsilon_{c}^+}\mathbf{1}\{\upsilon_{c}^+<\upsilon_{b}^-\}] \\
&\ \ \  +E_x\big[e^{-q \upsilon_{b}^-}\mathbf{1}\{\upsilon_{b}^-<\upsilon_{c}^+\}E_{Y_{\upsilon_{b}^-}}[e^{-q k_{b}^+}\mathbf{1}\{k_{b}^+<k_0^-\}]\big]
		\cdot E_{b}[e^{-q k_{c}^+}\mathbf{1}\{k_{c}^+<T\}]
		\nonumber \\
		& \ \ \ +E_x\Big[e^{-q \upsilon_{b}^-}\mathbf{1}\{\upsilon_{b}^-<\upsilon_{c}^+\}E_{Y_{\upsilon_{b}^-}}\big[e^{-q \tau_0^-}\mathbf{1}\{\tau_0^-<\tau_{b}^+\}E_{X_{\tau_0^-}}[e^{-q \tau_0^+}\mathbf{1}\{\tau_0^+<\tau_{-l}^-\wedge\xi\}]\big]\Big]\nonumber \\
		 &  \ \ \ \ \cdot E_0[e^{-q k_{c}^+}\mathbf{1}\{k_{c}^+<T\}].\nonumber
	\end{align}
	For $-l\leq x<b$, similarly, using the strong Markov property and considering the fact that $\{X_t, t<\tau_{b}^+\}$ and $\{R_t, t<k_{b}^+\}$ have the same distribution, we have
\begin{align}\label{13}
		&E_x[e^{-q k_{c}^+}\mathbf{1}\{k_{c}^+<T\}]
=E_x[e^{-q \tau_{b}^+}\mathbf{1}\{\tau_{b}^+<\tau_0^-\}]
		\cdot E_{b}[e^{-q k_{c}^+}\mathbf{1}\{k_{c}^+<T\}]
		\\ &\ \ \ +E_{x}\big[e^{-q \tau_0^-}\mathbf{1}\{\tau_0^-<\tau_{b}^+\}E_{X_{\tau_0^-}}[e^{-q \tau_0^+}\mathbf{1}\{\tau_0^+<\tau_{-l}^-\wedge\xi\}]\big]
		\cdot E_0[e^{-q k_{c}^+}\mathbf{1}\{k_{c}^+<T\}].	\nonumber
	\end{align}
Worth emphasizing, \eqref{12} is equivalently converted into \eqref{13} when $x<b$.
So then, \eqref{12} holds for $x\in[-l, c)$.

We now analyze the expressions involved in \eqref{12}-\eqref{13}. Since $\xi$ is independent of $x$, $\tau_0^+$ and $\tau_{-l}^-$, by \eqref{3} and \eqref{exp0xi}, we have, for $x<0$,
	\begin{align}\label{16}
		& E_x[e^{-q \tau_0^+}\mathbf{1}\{\tau_0^+<\tau_{-l}^-\wedge \xi\}]=E_x[e^{-(q+m) \tau_0^+}\mathbf{1}\{\tau_0^+<\tau_{-l}^-\}]
		=\frac{W^{(q+m)}(x+l)}{W^{(q+m)}(l)}.
	\end{align}
For $x\in [-l, b)$, following from Eqs. \eqref{6} and \eqref{16} that
	\begin{align}\label{17}
		&E_{x}\big[e^{-q \tau_0^-}\mathbf{1}\{\tau_0^-<\tau_{b}^+\}E_{X_{\tau_0^-}}[e^{-q \tau_0^+}\mathbf{1}\{\tau_0^+<\tau_{-l}^-\wedge\xi\}]\big]\nonumber\\
       & =\frac{1}{W^{(q+m)}(l)}[g^{(q+m,q)}(x,l)-\frac{W^{(q)}(x)}{W^{(q)}(b)}g^{(q+m,q)}(b,l)].
	\end{align}
For $x\in [b, c)$, on one hand, by Eqs. \eqref{5}, \eqref{8} and \eqref{10}, we obtain
	\begin{align}\label{15}
		& E_x\big[e^{-q \upsilon_{b}^-}\mathbf{1}\{\upsilon_{b}^-<\upsilon_{c}^+\}E_{Y_{\upsilon_{b}^-}}[e^{-q k_{b}^+}\mathbf{1}\{k_{b}^+<k_0^-\}]\big]
		=\frac{w^{(q)}(x;0)}{W^{(q)}(b)}-\varpi^{(q)}(x,b,c),
	\end{align}
on the other hand, through variable substitution, and using Fubini's theorem to interchange the order of expectation and integration operators, along with considering the spatial homogeneity of $Y$, by Eqs. \eqref{2}, \eqref{7}, \eqref{8} and \eqref{17}, we derive
	\begin{align}\label{18}
		& E_x\Big[e^{-q \upsilon_{b}^-}\mathbf{1}\{\upsilon_{b}^-<\upsilon_{c}^+\}E_{Y_{\upsilon_{b}^-}}\big[e^{-q \tau_0^-}\mathbf{1}\{\tau_0^-<\tau_{b}^+\}E_{X_{\tau_0^-}}[e^{-q \tau_0^+}\mathbf{1}\{\tau_0^+<\tau_{-l}^-\wedge\xi\}]\big]\Big]\nonumber		
        \\ &=\frac{1}{W^{(q+m)}(l)}\big(
        E_x[e^{-q\upsilon_{b}^-}\mathbf{1}\{\upsilon_{b}^-<\upsilon_{c}^+\}W^{(q)}(Y_{\upsilon_{b}^-}+l)]
		\nonumber		
        \\ &\ \ \ \ \ \ \ \ \ \ \ \ \ \ \ \ \ \ \  +m \int_{0}^{l}		E_x[e^{-q\upsilon_{b}^-}\mathbf{1}\{\upsilon_{b}^-<\upsilon_{c}^+\}W^{(q)}(Y_{\upsilon_{b}^-}+l-y)]W^{(q+m)}(y)\mathrm{d}y\nonumber
		\\ &\ \ \ \ \ \ \ \ \ \ \ \ \ \ \ \ \ \ \ -\frac{g^{(q+m,q)}(b,l)}{W^{(q)}(b)}E_x[e^{-q\upsilon_{b}^-}\mathbf{1}\{\upsilon_{b}^-<\upsilon_{c}^+\}W^{(q)}(Y_{\upsilon_{b}^-})]\big)
		\nonumber\\ &
=\frac{1}{W^{(q+m)}(l)}\Big(
        E_{x+l}[e^{-q\upsilon_{b+l}^-}\mathbf{1}\{\upsilon_{b+l}^-<\upsilon_{c+l}^+\}
        W^{(q)}(Y_{\upsilon_{b+l}^-})]\nonumber
		\\ &\ \ \ \ \ \ \ \ \ \ \ \ \ \ \ \ \ \ \ +m \int_{0}^{l}		E_{x+l-y}[e^{-q\upsilon_{b+l-y}^-}\mathbf{1}\{\upsilon_{b+l-y}^-<\upsilon_{c+l-y}^+\}W^{(q)}(Y_{\upsilon_{b+l-y}^-})]W^{(q+m)}(y)\mathrm{d}y\nonumber
		\\ &\ \ \ \ \ \ \ \ \ \ \ \ \ \ \ \ \ \ \ -\frac{g^{(q+m,q)}(b,l)}{W^{(q)}(b)}\big(w^{(q)}(x; 0) -\frac{\mathbb{W}^{(q)}(x-b)}{\mathbb{W}^{(q)}(c-b)}w^{(q)}(c; 0) \big)\Big)
		\nonumber\\ &
		=\frac{1}{W^{(q+m)}(l)}\Big(\vartheta^{(q+m,q)}(x, l)-\frac{\mathbb{W}^{(q)}(x-b)}{\mathbb{W}^{(q)}(c-b)}\vartheta^{(q+m,q)}(c, l) \nonumber \\ &\ \ \ \ \ \ \ \ \ \ \ \ \ \ \ \ \ \ \ -\big(\frac{w^{(q)}(x;0)}{W^{(q)}(b)}-\varpi^{(q)}(x,b,c)\big)g^{(q+m,q)}(b,l)\Big),
	\end{align}
	where $\vartheta^{(q+m,q)}(x, l)$ and $\varpi^{(q)}(x,b,c)$ are given in Eqs. \eqref{9} and \eqref{10} respectively.
Accordingly, the only remaining unknown terms in \eqref{12}-\eqref{13} are $E_0[e^{-q k_{c}^+}\mathbf{1}\{k_{c}^+<T\}]$ and $E_b[e^{-q k_{c}^+}\mathbf{1}\{k_{c}^+<T\}]$.

Assume that $X$ has paths of BV. In this case, $W^{(q)}(0)\neq 0$, and thus, setting $x=b$ into \eqref{12} and $x=0$ into \eqref{13} yields a system of equations. By Eqs. \eqref{3}, \eqref{4} and \eqref{16}-\eqref{18}, we solve this system to obtain
	\begin{eqnarray}
		&&E_{b}[e^{-q k_{c}^+}\mathbf{1}\{k_{c}^+<T\}]=\frac{g^{(q+m,q)}(b, l)}{\vartheta^{(q+m,q)}(c, l)},\label{19}
		\\&&E_0[e^{-q k_{c}^+}\mathbf{1}\{k_{c}^+<T\}]=\frac{W^{(q+m)}(l)}{\vartheta^{(q+m,q)}(c, l)}.\label{20}
	\end{eqnarray}
The case where $X$ has paths of UBV follows using the same approximating procedure as in Loeffen, Renaud and Zhou\cite{Loeffen2014}, and the result is consistent with \eqref{19}-\eqref{20}.
Finally, for $x\in[-l, c)$, plugging Eqs. \eqref{19} and \eqref{20} into \eqref{12}, and using Eqs. \eqref{4}, \eqref{15} and \eqref{18}, we can get the result in \eqref{11}.
\end{proof}

\subsection{Proof of Corollary \ref{cor00}}
\begin{proof}\label{cor2100}
By \eqref{1wqde}, \eqref{2} and \eqref{9}, we have
\begin{align*}
&\vartheta^{(q+m,q)}(x,-l)\\
&= W^{(q+m)}(x+l) - m \int_0^x W^{(q)}(x-y) W^{(q+m)}(l+y) \mathrm{d} y +\delta \mathbf{1}{\{x\geq b\}}
 \int_b^x \mathbb{W}^{(q)}(x-y) \\
& \ \ \ \cdot \Big(W^{(q+m)}{'}(y+l) - m\big(W^{(q)}(0) W^{(q+m)}(y+l) + \int_0^y W^{(q)}{'}(y-z) W^{(q+m)}(z+l)\mathrm{d}z\big) \Big)\mathrm{d}y,
\end{align*}
since
\begin{align*}
&W^{(q)}{'}(y+l) + m \int_0^l W^{(q)}{'} (y+l-z) W^{(q+m)}(z) \mathrm{d} z \\
&=\frac{\mathrm{d}}{\mathrm{d}y}\bigl(W^{(q+m)}(y+l) - m \int_0^y W^{(q)} (y-z) W^{(q+m)}(z+l) \mathrm{d} z \bigr) \\
&=W^{(q+m)}{'}(y+l) - m \bigl( W^{(q)}(0) W^{(q+m)}(y+l) + \int_0^y W^{(q)}{'}(y-z) W^{(q+m)}(z+l) \mathrm{d}z \bigr).
\end{align*}
From Kyprianou\cite{Kyprianou20142} it follows that $\lim_{l\to\infty}\tfrac{W^{(q)}(x+l)}{W^{(q)}(l)} =e^{\phi(q)x}$ and $\lim_{l\to \infty}$ $\tfrac{{W^{(q)}}{'}(x+l)}{W^{(q)}(l)}=\phi(q)e^{\phi(q)x}$.
Dividing both the numerator and the denominator in the above expression by
$W^{(q+m)}(l)$ and then letting $l\to\infty$, we obtain \eqref{ex26}.
\end{proof}

\subsection{Proof of Lemma \ref{lemdiff}}
\begin{proof}\label{appen}
According to the It\^{o}-L\'{e}vy decomposition (see Kyprianou\cite{Kyprianou20142}), the uncontrolled surplus process $X$ evolves as
\begin{align*}
&X_t = \gamma t + \sigma B_t - \int_0^t \int_{z\in(0,1)} z \tilde{M}(\mathrm{d}s, \mathrm{d}z) - \int_0^t \int_{z\in[1,\infty)} z M(\mathrm{d}s, \mathrm{d}z),
\end{align*}
while the controlled process $U^{\pi}$ takes the form
\begin{align*}
U_t^{\pi} &= (\gamma t - \delta \int_0^t \mathbf{1}\{U_s^{\pi}\geq b\} \mathrm{d} s)
 + \sigma B_t - \int_0^t \int_{z\in(0,1)} z \tilde{M}(\mathrm{d}s, \mathrm{d}z) - \int_0^t \int_{z\in[1,\infty)} z M(\mathrm{d}s, \mathrm{d}z)\\
 &\ \ \  - D_t^{\pi},
\end{align*}
where $B_t$ is an independent Brownian Motion with a zero mean and a variance coefficient of 1, the finite variation process $D^{\pi}$ is regarded as the dividend process, $M(\mathrm{d}t, \mathrm{d}z)$ is a Poisson random measure of $U_t^{\pi}$ (or $X_t$) with characteristic measure $\nu(\mathrm{d}z)\mathrm{d}t$, and $\tilde{M}(\mathrm{d}t, \mathrm{d}z)= M(\mathrm{d}t, \mathrm{d}z) - \nu(\mathrm{d}z)\mathrm{d}t$ is the corresponding compensated Poisson random measure. It is well-known that for each Borel set $A\subset [a, \infty)$ with some $a>0$, $M(t, A)=M([0, t] \times A) = \sum_{s\geq 0}^t \mathbf{1}_A\{s, X_{s-}^{\pi} -X_s^{\pi}\}$, and its compensated process $\tilde{M}(t, A)$ is a martingale. Note that $\Delta M_t := \int_{(0, \infty)} z M(\mathrm{d}t, \mathrm{d} z)$ denote the total jump magnitude at time $t$, so that $\Delta X_t = -\Delta M_t$. Accordingly, the differential form of $U_t^{\pi}$ is given by
\begin{align}\label{utpi561}
&\mathrm{d} U_t^{\pi} = (\gamma - \delta \mathbf{1}\{U_t^{\pi}\geq b\}) \mathrm{d} t
 + \sigma \mathrm{d} B_t - \int_{z\in(0,1)} z \tilde{M}(\mathrm{d}t, \mathrm{d}z) - \int_{z\in[1,\infty)} z M(\mathrm{d}t, \mathrm{d}z) - \mathrm{d} D_t^{\pi}.
\end{align}
Let $f: \mathbb{R}\rightarrow \mathbb{R}$ be a sufficiently smooth function. $M(t, A)$ and $\tilde{M}(t, A)$ can also be interpreted as
\begin{align*}
& \int_{0}^{t}\int_{z\in[a, \infty)}f(z)M(\mathrm{d}s, \mathrm{d}z)=\sum_{j\ge1}f(\xi_{j})\mathbf{1}\{\xi_{j}\geq a\} \mathbf{1}\{\tau_{j}\le t\},\\
& \int_{0}^{t}\int_{z\in[a, \infty)}f(z)\tilde{M}(\mathrm{d}s,\mathrm{d}z)=
\sum_{j\ge1}\big(f(\xi_{j})-E[f(\xi)]\big)\mathbf{1}\{\xi_{j}\geq a\} \mathbf{1}\{\tau_{j}\le t\}.
\end{align*}
By applying the Bouleau-Yor formula to the BV components (including the control-induced process $D^{\pi}$), the Meyer-It\^{o} formula to the continuous UBV component, and the It\^{o}-L\'{e}vy formula to the jump terms driven by the Poisson random measure $M$, we obtain
\begin{align*}
&e^{-qt - L^{\pi}(t)}f(U_{t}^{\pi})
\nonumber \\
&=f(x)-\int_0^{t}(q+m \mathbf{1}\{U_{s-}^{\pi}<0\}) e^{-q s - L^{\pi}(s)} f(U_{s-}^{\pi})\mathrm{d}s
		+\int_0^{t}e^{-q s- L^{\pi}(s)}f{'}(U_{s-}^{\pi})\mathrm{d}U_{s}^{\pi}\\
&  +\frac{\sigma^2}{2}\int_0^{t}e^{-q s - L^{\pi}(s)} f{''}(U_{s-}^{\pi})\mathrm{d}s
 +\sum_{0\leq s\leq t}e^{-q s - L^{\pi}(s)}\big(f(U_{s-}^{\pi}+\Delta U_s^{\pi})-f(U_{s-}^{\pi})
-f{'}(U_{s-}^{\pi})\Delta U_s^{\pi}\big),\nonumber
\end{align*}
where the final term denotes the summand over the jumps before $t$, and $\Delta U_t^{\pi}= U_t^{\pi}-U_{t-}^{\pi}$ represents the jump of the controlled process $U^{\pi}$ at time $t$. Taking into account that the final term arises either from dividend-induced jumps $\Delta D_t^{\pi}$ with $\Delta D_{t}^{\pi}=D_t^{\pi}-D_{t-}^{\pi}$, from Poisson-driven jumps $\Delta M_t = -\Delta X_t$, or from the simultaneous occurrence of both, which implies $\Delta U_t = -\Delta M_t- \Delta D_t^{\pi} = \Delta X_t- \Delta D_t^{\pi}$, and combining this with \eqref{utpi561}, we obtain
\begin{align*}
&e^{-qt - L^{\pi}(t)}f(U_{t}^{\pi})
= f(x)-\int_0^{t}(q+m \mathbf{1}\{U_{s-}^{\pi}<0\})e^{-q s - L^{\pi}(s)} f(U_{s-}^{\pi})\mathrm{d}s
		\\
 &  \ \ \ \ \ \ \ \ \ \ \ \ \ \ \ \  +\int_0^{t}e^{-q s -L^{\pi}(s)}f{'}(U_{s-}^{\pi})
(\gamma - \delta \mathbf{1}\{U_{s-}^{\pi}\geq b\}) \mathrm{d} s
 + \mathcal{M}_{t}^{(1)} \\
&  \ \ \ \ \ \ \ \ \ \ \ \ \ \ \ \
 - \int_0^{t}e^{-q s - L^{\pi}(s)}f{'}(U_{s-}^{\pi}) \big(\int_{z\in(0,1)} z \tilde{M}(\mathrm{d}s, \mathrm{d}z) + \int_{z\in[1,\infty)} z M(\mathrm{d}s, \mathrm{d}z)\big)
\\
 &  \ \ \ \ \ \ \ \ \ \ \ \ \ \ \ \  - \sum_{0\leq s\leq t}e^{-q s - L^{\pi}(s)} f{'}(U_{s-}^{\pi})\Delta D_{s}^{\pi} +\frac{\sigma^2}{2}\int_0^{t}e^{-q s - L^{\pi}(s)} f{''}(U_{s-}^{\pi})\mathrm{d}s
	  \\
& \ \ \ \ \ \ \ \ \ \ \ \ \ \ \ \ +\sum_{0\leq s\leq t}e^{-q s - L^{\pi}(s)}\big(f(U_{s-}^{\pi} - \Delta M_{s} - \Delta D_s^{\pi})-f(U_{s-}^{\pi})+ f{'}(U_{s-}^{\pi})(\Delta M_{s} +\Delta D_{s}^{\pi})\big),
\end{align*}
where $\mathcal{M}_{t}^{(1)} = \int_0^{t}e^{-q s - L^{\pi}(s)}f{'}(U_{s-}^{\pi}) \sigma \mathrm{d} B_s$. Further simplification and consolidation yields
\begin{align}\label{dfupi1}
 &e^{-qt - L^{\pi}(t)}f(U_{t}^{\pi})= f(x)-\int_0^{t}(q + m\mathbf{1}\{U_{s-}^{\pi}<0\}) e^{-q s - L^{\pi}(s)} f(U_{s-}^{\pi})\mathrm{d}s
		 \\
&  +\int_0^{t}e^{-q s - L^{\pi}(s)}f{'}(U_{s-}^{\pi})
(\gamma - \delta \mathbf{1}\{U_{s-}^{\pi}\geq b\}) \mathrm{d} s
 + \mathcal{M}_{t}^{(1)} +\frac{\sigma^2}{2}\int_0^{t}e^{-q s - L^{\pi}(s)} f{''}(U_{s-}^{\pi})\mathrm{d}s
+ \mathcal{M}_{t}^{(2)}\nonumber \\
&
+ \sum_{0\leq s\leq t}e^{-q s - L^{\pi}(s)} \big(f(U_{s-}^{\pi} - \Delta M_{s} - \Delta D_s^{\pi})-f(U_{s-}^{\pi})+ f{'}(U_{s-}^{\pi})\Delta M_{s}\mathbf{1}\{0< \Delta M_{s}< 1\} \big),\nonumber
\end{align}
where $\mathcal{M}_{t}^{(2)} = - \int_0^{t}e^{-q s - L^{\pi}(s)}f{'}(U_{s-}^{\pi}) \int_{z\in(0,1)} z \tilde{M}(\mathrm{d}s, \mathrm{d}z)$.
The final term in \eqref{dfupi1} can be equivalently expressed as
\begin{align}\label{dfupi2}
&\sum_{0\leq s\leq t}e^{-q s - L^{\pi}(s)} \big(f(U_{s-}^{\pi} - \Delta M_{s} - \Delta D_s^{\pi})-f(U_{s-}^{\pi})+ f{'}(U_{s-}^{\pi})\Delta M_{s}\mathbf{1}\{0< \Delta M_{s}< 1\} \big) \nonumber\\
&= \sum_{0\leq s\leq t}e^{-q s - L^{\pi}(s)} \big(f(U_{s-}^{\pi} - \Delta M_{s} - \Delta D_s^{\pi})-f(U_{s-}^{\pi} - \Delta M_{s})
+f(U_{s-}^{\pi} - \Delta M_{s}) - f(U_{s-}^{\pi}) \nonumber\\
& \ \ \ \ \ \ \ \ \ \ \ \ \ \ \ \ \ \ \ \ \ \ \ \ + f{'}(U_{s-}^{\pi})\Delta M_{s}\mathbf{1}\{0< \Delta M_{s}< 1\} \big)\nonumber\\
& = \sum_{0\leq s\leq t}e^{-q s - L^{\pi}(s)} \big(f(U_{s-}^{\pi} - \Delta M_{s} - \Delta D_s^{\pi})-f(U_{s-}^{\pi} - \Delta M_{s})\big)
 +\int_0^{t}e^{-q s - L^{\pi}(s)} \nonumber\\
& \ \ \ \cdot \int_{(0, \infty)} \big(f(U_{s-}^{\pi} - z)-f(U_{s-}^{\pi})+ f{'}(U_{s-}^{\pi})z\mathbf{1}\{0<z<1\}\big) \big(\tilde{M}(\mathrm{d}s, \mathrm{d} z) + \nu(\mathrm{d} z) \mathrm{d} s\big)\nonumber\\
& = \sum_{0\leq s\leq t}e^{-q s - L^{\pi}(s)} \big(f(U_{s-}^{\pi} - \Delta M_{s} - \Delta D_s^{\pi})-f(U_{s-}^{\pi} - \Delta M_{s})\big) +
\mathcal{M}_{t}^{(3)} \nonumber\\
& \ + \int_0^{t}e^{-q s - L^{\pi}(s)} \int_{(0, \infty)} \big(f(U_{s-}^{\pi} - z)-f(U_{s-}^{\pi})+ f{'}(U_{s-}^{\pi})z\mathbf{1}\{0<z<1\}\big) \nu(\mathrm{d} z) \mathrm{d} s,
\end{align}
where $\mathcal{M}_{t}^{(3)} = \int_0^{t}e^{-q s - L^{\pi}(s)} \int_{(0, \infty)} \big(f(U_{s-}^{\pi} - z)-f(U_{s-}^{\pi})+ f{'}(U_{s-}^{\pi})z \mathbf{1}\{0<z<1\}\big) \tilde{M}(\mathrm{d}s, \mathrm{d} z)$.
Substituting \eqref{dfupi2} into \eqref{dfupi1} results in
\begin{align}
& e^{-qt - L^{\pi}(t)}f(U_{t}^{\pi})= f(x)-\int_0^{t}(q+m \mathbf{1}\{U_{s-}^{\pi}<0\}) e^{-q s - L^{\pi}(s)} f(U_{s-}^{\pi})\mathrm{d}s \nonumber \\
&
 +\int_0^{t}e^{-q s - L^{\pi}(s)}
(\gamma - \delta \mathbf{1}\{U_{s-}^{\pi}\geq b\}) f{'}(U_{s-}^{\pi}) \mathrm{d} s
+ \mathcal{M}_{t}
  \nonumber
  \\ & + \frac{\sigma^2}{2} \int_0^t e^{-q s - L^{\pi}(s)} f{''}(U_{s-}^{\pi}) \mathrm{d} s
  +\sum_{0\leq s\leq t}e^{-q s - L^{\pi}(s)} \big(f(U_{s-}^{\pi} - \Delta M_{s} - \Delta D_s^{\pi})-f(U_{s-}^{\pi} - \Delta M_{s})\big)\nonumber
  \\ &
+ \int_0^{t}e^{-q s - L^{\pi}(s)} \int_{(0, \infty)} \big(f(U_{s-}^{\pi} - z)-f(U_{s-}^{\pi})+ f{'}(U_{s-}^{\pi})z \mathbf{1}\{0<z<1\}\big) \nu(\mathrm{d} z) \mathrm{d} s,\nonumber
\end{align}
where
\begin{align}\label{martinm}
&\mathcal{M}_{t} =  \mathcal{M}_{t}^{(1)} + \mathcal{M}_{t}^{(2)} + \mathcal{M}_{t}^{(3)} \\
& = \int_0^{t}e^{-q s - L^{\pi}(s)} \sigma f{'}(U_{s-}^{\pi}) \mathrm{d} B_s
+ \int_0^{t}e^{-q s - L^{\pi}(s)} \int_{(0, \infty)} \big(f(U_{s-}^{\pi} - z)-f(U_{s-}^{\pi}) \big) \tilde{M}(\mathrm{d}s, \mathrm{d} z).\nonumber
\end{align}
Since $B_t$ and $\tilde{M}(t, A)$ are both martingales, it naturally follows that $\mathcal{M}_{t}$ is a zero-mean $P_x$-martingale. By the operator $\mathcal{A}$ defined in \eqref{operator} and $\Delta X_t = -\Delta M_t$, the above expression is equivalent to \eqref{dududu}.
\end{proof}

\end{document}